\documentclass[preprint,reqno]{imsart}
\usepackage[margin=1.25in]{geometry}

\usepackage[utf8]{inputenc}
\usepackage{amsmath}
\usepackage{amssymb}
\usepackage{graphicx}
\usepackage[numbers,sort&compress]{natbib}
\usepackage{bbm}
\usepackage{appendix}
\usepackage{amsthm}
\usepackage{xcolor}
\usepackage{tabularx,multicol,multirow,booktabs,csquotes,comment,mathtools}

\usepackage{verbatim}
\usepackage{hyperref}
\hypersetup{
	colorlinks,
	linkcolor={red!50!black},
	citecolor={blue},
	urlcolor={blue!80!black}
}

\RequirePackage{hypernat}
\usepackage{bm}
\usepackage[shortlabels]{enumitem}
\usepackage{subfigure}
\usepackage[noabbrev,capitalize]{cleveref}

\allowdisplaybreaks

\newtheorem{thm}{Theorem}[section]
\newtheorem{cor}[thm]{Corollary}
\newtheorem{defi}{Definition}[section]
\newtheorem{assume}{Assumption}[section]

\newtheorem{lemma}[thm]{Lemma}
\newtheorem{ex}{Example}[section]
\newtheorem{remark}{Remark}[section]

\numberwithin{equation}{section}

\DeclareMathOperator{\Cov}{Cov}
\DeclareMathOperator{\Var}{Var}

\newcommand{\pQ}{\mathbb{Q}^{\mathrm{prod}}}
\newcommand{\R}{\mathbb{R}}
\newcommand{\I}{\mathbf{I}}

\DeclareMathOperator{\an}{\alpha_{\mathnormal n}}
\DeclareMathOperator{\be}{\mathbf e}
\DeclareMathOperator{\bmm}{\mathbf m}

\DeclareMathOperator{\bz}{\mathbf z}
\DeclareMathOperator{\bs}{\bolds}

\DeclareMathOperator{\sech}{sech}
\DeclareMathOperator{\EE}{\mathbb{E}}

\DeclareMathOperator{\PP}{\mathbb P}
\DeclareMathOperator{\QQ}{\mathbb Q}

\DeclareMathOperator*{\xp}{\xrightarrow{\mathnormal P}}
\DeclareMathOperator*{\xd}{\xrightarrow{\mathnormal d}}

\newcommand{\bgamma}{\bm{\gamma}}

\DeclareMathOperator{\mcn}{N}

\newcommand{\A}{\mathbf A}
\newcommand{\G}{\mathbf G}
\newcommand{\bb}{\mathbf b}
\newcommand{\bc}{\mathbf c}
\newcommand{\bq}{\mathbf q}
\newcommand{\bu}{\mathbf u}

\newcommand{\bv}{\mathbf v}
\newcommand{\bw}{\mathbf w}

\newcommand{\by}{\mathbf y}

\newcommand{\C}{\mathbf C}
\newcommand{\M}{\mathbf M}

\newcommand{\bolds}{\mathnormal{\boldsymbol \sigma}}
\newcommand{\boldc}{\mathnormal{\mathbf{c}}}
\newcommand{\sumin}{\sum_{\mathnormal i=1}^{\mathnormal n}}
\newcommand{\maxin}{\max_{\mathnormal i=1}^{\mathnormal n}}
\newcommand{\sumjn}{\sum_{\mathnormal j=1}^{\mathnormal n}}
\newcommand{\sumij}{\sum_{\mathnormal{ i,j}=1}^{\mathnormal n}}

\newcommand{\red}[1]{#1}

\begin{document}
	\begin{frontmatter}
		\title{Fluctuations in random field Ising models}
		\runtitle{Fluctuations in random field Ising models}
        \runauthor{Lee et al.}
		\thankstext{m3}{The research of Sumit Mukherjee was supported in part by NSF Grant DMS-2113414.}
		
		 \begin{aug}
			\author{\fnms{Seunghyun} \snm{Lee}\ead[label=e1]{ sl4963@columbia.edu}},
			\author{\fnms{Nabarun} \snm{Deb}\ead[label=e2]{nabarun.deb@chicagobooth.edu}},
		 	\and
		 	\author{\fnms{Sumit} \snm{Mukherjee}\thanksref{m3}\ead[label=e3]{sm3949@columbia.edu}}
		 \end{aug}
        \address{Department of Statistics, Columbia University\printead[presep={,\ }]{e1,e3}}
        \address{Chicago Booth School of Business\printead[presep={,\ }]{e2}}
  
		\begin{abstract}
  This paper establishes a  CLT for linear statistics of the form $\langle \bq,\bolds\rangle$ with quantitative Berry-Esseen bounds, where $\bolds$ is an observation from an exponential family with a quadratic form as its sufficient statistic, in the \enquote{high-temperature} regime. We apply our general result to random field Ising models with both discrete and continuous spins. To demonstrate the generality of our techniques, we apply our results to derive both quenched and annealed CLTs in various examples, which include Ising models on some graph ensembles of common interest (Erd\H{o}s-R\'{e}nyi, regular, random graphons), and the Hopfield spin glass model. Our proofs rely on a combination of Stein's method of exchangeable pairs and Chevet-type concentration  inequalities.
  \end{abstract}
		
		
	\end{frontmatter}
	
	\maketitle
	
	\section{Introduction}
	
	Let $\bolds:=(\sigma_1,\ldots , \sigma_n)^{\red{\top}} \in [-1,1]^n$ be a random vector drawn from the following quadratic interaction model with external field $\boldc:=(c_1,\ldots ,c_n)^{\red{\top}} \in \R^n$:
	\begin{equation}\label{eq:model}
		\frac{d\PP}{d\prod_{i=1}^n \mu_i}(\bolds):=\frac{1}{Z_n}\exp\left(\frac{1}{2}\bolds^{\top} \A_n\bolds+\boldc^{\top}\bolds\right),  
	\end{equation}
    where
    \begin{align}\label{eq:normconst}
     Z_n:=\int \exp\left(\frac{1}{2}\bolds^{\top} \A_n\bolds+\boldc^{\top}\bolds\right)\prod_{i=1}^n\,d\mu_i(\sigma_i).
     \end{align}
	Here, each $\mu_i$ is a \red{non-degenerate} base-measure supported on $[-1,1]$ \red{where both $\{-1,1\}$ belong in the support of each $\mu_i$}, and $\A_n$ is a \red{real} symmetric matrix with zeros on the diagonal.
    We will refer to $\bc$ as the field vector, and the factor $Z_n$ as the normalizing constant/partition function of model \eqref{eq:model}.
    Note that the inverse temperature parameter for typical Gibbs measures is absorbed into the coupling matrix $\A_n$.
    We omit the dependence of $(\A_n, \bc, \{\mu_i\}_{1\le i\le n})$ on the measure $\PP$ and normalizing constant $Z_n$ for simplicity of notation. 
    
    The goal of this paper is to study limit theorems for linear combinations of spins $\bolds$, i.e., 
	\begin{align}\label{eq:targetstat}
		T_n:=\bq^{\top}\bolds=\sum_{i=1}^n q_i\sigma_i,
	\end{align}
	for $\bq:=(q_1,\ldots ,q_n)^\top \in\mathbb{R}^n$, normalized to have $\|\bq\|=1$, when the matrix $\A_n$ is in the so called \emph{high-temperature Mean-Field} regime of statistical physics (see Assumptions \ref{assn:ht} and \ref{assn:mf} below).
 Our main motivation for studying limit theorems under \eqref{eq:model} comes from the following example: 

    	\begin{ex}[Random field Ising model]\label{ex:rfim}
		Consider an observation from model \eqref{eq:model} where the base measures $\mu_i=\mu$ are the same, and the field vector $\bc=(c_1,\ldots ,c_n)^\top$ is a fixed realization of $n$ i.i.d. random variables from some probability distribution, say $F$. The most studied variant of this model is the random field Ising model (RFIM, {sometimes referred to as the disordered Ising model}), where $\mu_i$s are supported on $\{-1,1\}$ and $\A_n$ represents the (scaled) adjacency matrix of a  graph. Two commonly studied graph ensembles are the complete graph (see \cite{salinas1985mean,de1991fluctuations,he2023hidden,lowe2023propagation}), and a neighborhood graph which is a subgraph of the $d$-dimensional lattice (see \cite{chatterjee2019central, chatterjee2015absence,ding2024long, chatterjee2023features, ding2021exponential, ding2023new, ding2024long2}). 
        In contrast, we look at more general spin-glass models, where the entries of $\A_n$ can have both positive and negative entries, and also allow for general compactly supported probability measures $\mu$ (discrete or continuous). Do note that our results apply only in  the so called high-temperature regime (see~\cref{assn:ht}).
    \end{ex}

\subsection{Main Contributions}\label{sec:contrib}
	
	Our main contributions are two-fold. 
 \begin{itemize}
 
 \item[(a)] In \cref{sec:clt}, we provide our main results regarding concentration and fluctuations of the linear statistic $T_n = \bq^\top \boldsymbol{\sigma}$, where $\bolds$ is a random sample from the distribution $\PP$ (see \eqref{eq:model}). This includes the following: (i) Law of large numbers (LLN) and concentration bounds (in \cref{thm:LLN}), (ii) a CLT with finite sample Berry-Esseen bounds with a somewhat implicit centering %
 (in \cref{thm:clt}), 
 and (iii) another CLT with Berry-Esseen bounds with a more explicit centering, under mildly stronger assumptions
 (see \cref{thm:CLT for u}). The LLN is stated under minimal assumptions, where $\bq$ is an arbitrary vector. For the Berry-Esseen bounds to converge to $0$, we additionally require $\bq$ to be an approximate delocalized eigenvector of the coupling matrix $\A_n$. In fact, the corresponding approximate eigenvalue shows up in the limiting normal distribution. We note that as our results are non-asymptotic with explicit error bounds, they can be readily applied to provide quenched and annealed limits for $T_n$ in a variety of quadratic interaction models, as illustrated in \cref{cor:i.i.d. clt}.

   \item[(b)] Next, in \cref{sec:isingrandom}, we apply our results in the context of the random field Ising model, introduced in \cref{ex:rfim}. To the best of our knowledge, only the average magnetization of the Curie-Weiss (complete graph) RFIM has been understood in the literature \citep{de1991fluctuations}. In \cref{cor:i.i.d. clt}, we state our main result for RFIMs with general coupling matrices $\A_n$, which include adjacency matrices of several graph ensembles  (Erd\H{o}s-Rényi graphs,  regular graphs, random graphons), as well as  the celebrated spin-glass Hopfield model \citep{hopfield1982neural}. Moving beyond complete graphs brings in new technical challenges which we overcome with new proof techniques that involve Stein's method of exchangeable pairs \cite{chatterjee2011nonnormal,Shao2019}, new Chevet-type inequalities (see \cref{lem:L_n norm}),  and new contraction bounds on certain ``local" averages (see \eqref{eq:locfield} and \cref{lem:m-n contraction}).
\end{itemize}

\subsection{Notation}\label{sec:notation}
	
	For two measures $\pi_1,\pi_2$ on the same probability space, define the Kullback-Leibler divergence between $\pi_1$ and $\pi_2$ as
 $$ \mathrm{KL}(\pi_1|\pi_2) := \begin{cases} \EE_{\pi_1}\log{\frac{d\pi_1}{d\pi_2}} & \mbox{if } \pi_1 \ll \pi_2, \\  \infty & \mbox{otherwise}.\end{cases} $$
 \noindent Given an arbitrary symmetric $n\times n$ real valued matrix $\mathbf{M}_n$, let $\lVert \mathbf{M}_n\rVert$, $\lVert \mathbf{M}_n\rVert_{r}$, and $\lVert \mathbf{M}_n\rVert_F$ denote the $(2,2)$-operator norm, the $(r,r)$-operator norm for $r\in (2,\infty]$, and the Frobenius norm of $\mathbf{M}_n$ respectively. { We also let $\mathrm{Off}(\mathbf{M}_n)$ denote the off-diagonal operator that sets the diagonal entries in $\mathbf{M}_n$ to 0.}
    For a vector $\mathbf{a}=(a_1,\ldots ,a_n)^\top\in \R^n$, let $\lVert \mathbf{a}\rVert$ and $\lVert \mathbf{a}\rVert_{r}$ denote the  $\ell_2$ and  $\ell_r$-vector norms, for $r\in (2,\infty]$. Let $\mbox{Diag}(\mathbf{a})$ denote the $n\times n$ diagonal matrix with diagonal entries $a_1,\ldots ,a_n$. We denote the $n$ dimensional all-zero vector, all-one vector, and $n\times n$ identity matrix by $\mathbf{0}$, $\mathbf{1}$, and $\I_n$  respectively. For $x\in\R$, $\lceil x\rceil$ will denote the smallest integer larger than or equal to $x$. With $\A_n$ as in \eqref{eq:model}, define 
	\begin{align}\label{eq:rowcontrol}
		\alpha_n:=\max_{1\leq i\leq n} \sum_{j=1}^{\red{n}} \A_n^2(i,j).
	\end{align}
    For nonnegative sequences $\{a_n\}_{n\ge 1}$ and $\{b_n\}_{n\ge 1}$, we write $a_n \lesssim b_n$ or $a_n = O(b_n)$, if there exists a constant $C>0$ free of $n$ such that $a_n\le C b_n$. %
    We use $a_n\ll b_n$ or $a_n = o(b_n)$ if $a_n/b_n\to 0$ as $n\to\infty$. We also use the standard $o_P, O_p$ notations for a sequence of random variables. The notation $\overset{d}{\longrightarrow}$ denotes weak convergence/convergence in distribution, and $\overset{P}{\longrightarrow}$ denotes convergence in probability. %
    For  a random vector ${\bs}$ from the model~\eqref{eq:model}, we define the random vector of \emph{local fields} as $\mathbf{m}:=(m_1,\ldots ,m_n)^\top$ where 
    \begin{equation}\label{eq:locfield}
        m_i:=\sum_{j=1}^n \A_n(i,j)\sigma_j.
    \end{equation}
   Given random variables $X$ and $Y$ supported on $\R$, the Kolmogorov-Smirnov (KS) distance between their distributions is defined as:
    \begin{align}\label{eq:ksdist}
    d_{KS}(X,Y)=\sup_{t\in\R} \big| P(X\le t) - P(Y\le t)\big|.
    \end{align}
     
   \noindent Unless otherwise mentioned, all probabilities are computed under $\PP$. Finally, we introduce the notion of exponential tilting for probability measures below. 
    
\begin{defi}[Exponential tilting]\label{def:expfam}
	Consider the \red{non-degenerate} base measures $\mu_i$s in \eqref{eq:model}, for $1\le i\le n$.  For $\theta\in\R$, define the tilted measures $\mu_{i,\theta}$ as follows: 
        $$\frac{d\mu_{i,\theta}}{d\mu_i}(z)\propto \exp(\theta z).$$
        Let $\psi_i(\cdot)$ denote the log moment generating function of $\mu_i$, defined by
		$\psi_i(\theta):=\log \int e^{\theta z}d\mu_i(z)$
		for $\theta \in \R$. 
		Then, standard exponential family calculations gives
		$$\psi_i'(\theta)=\EE_{Z \sim \mu_{i,\theta}}[Z],\quad \psi_i''(\theta)={\rm Var}_{Z \sim \mu_{i,\theta}}[Z].$$
	In particular, $\psi_i''(\cdot)$ is strictly positive, so $\psi_i':\mathbb{R} \mapsto (-1,1)$ is strictly increasing. Let $\phi_i:(-1,1)\mapsto \mathbb{R}$ be the inverse of $\psi_i'(\cdot)$, i.e.~$\phi_i(\psi_i'(\theta))=\theta$ for all $\theta\in \mathbb{R}$. Finally, define the KL divergence between the tilt $\mu_{i,\phi_i(z)}$ and the base-measure $\mu_i$ as:
		$$I_i(z):=KL(\mu_{i,\phi_i(z)}|\mu_i)=z\phi_i(z) -\psi_i(\phi_i(z)).$$
	\end{defi}
       \section{Limit theory for quadratic interaction models}\label{sec:mainres}

    This Section presents general theory
    regarding the concentration and Gaussian fluctuations of the linear statistic $T_n=\bq^{\top}\bs$ (see \eqref{eq:targetstat}). Throughout this Section, we will use the general notation $\A_n,\bc$ from \eqref{eq:model} and $\mu_i,\psi_i$ from \cref{def:expfam}.

\subsection{Assumptions and Mean-Field approximation}\label{sec:assumptions}
We will make some assumptions on the interaction matrix $\A_n$. As we shall show in \cref{sec:isingrandom} below, these assumptions hold for a wide-variety of relevant examples, covering both deterministic and random matrices $\A_n$. 
	
	\begin{assume}[high-temperature]\label{assn:ht}
		We present a sequence of three high-temperature regime assumptions which are increasingly stronger as will be demonstrated in the subsequent lemma. Fix $\rho\in (0,1)$.
		\begin{itemize}
			\item Weak high-temperature regime (WHT) assumption.
			\begin{align}\label{eq:wht}
				\lVert \A_n\rVert\le \rho.
			\end{align}
			\item Moderate high-temperature regime (MHT) assumption.
			\begin{align}\label{eq:mht}
				\lVert \A_n\rVert_{4}\le \rho.
			\end{align}
			\item Strong high-temperature regime (SHT) assumption.
			\begin{align}\label{eq:sht}
				\lVert\A_n\rVert_{\infty}\le \rho.
			\end{align}
		\end{itemize}
  \noindent Depending on the result, we will invoke one of the three assumptions above.
	\end{assume}

Note that the SHT assumption in \eqref{eq:sht} is commonly referred to as the high-temperature Dobrushin uniqueness condition in the Gibbs distribution literature; see e.g. \cite[Equation 1.1]{Reza2018} and \cite[Equation 2.2]{Christof2003}. Different forms of such high-temperature assumptions on $\A_n$ have gained significant popularity in the recent years (see e.g., \cite{mukherjee2022variational,barbier2020mutual,qiu2024sub,mukherjee2024naive,celentano2023mean}). The following matrix lemma provides an ordering between the three high-temperature regimes stated above.  

\begin{lemma}\label{lem:normorder}
	Let $\mathbf{M}_n$ be a $n \times n$ symmetric matrix. Then, $\lVert \mathbf{M}_n\rVert \le \lVert \mathbf{M}_n\rVert_{r} \le \lVert \mathbf{M}_n\lVert_{\infty}$ for $r\in (2,\infty)$.
\end{lemma}

\noindent We refer the reader to \cref{sec:pfauxlem} for a proof. 

Next, we introduce the Mean-Field approximation of $\PP$ (see \eqref{eq:model}), which will play an instrumental role in describing the asymptotic distribution of $T_n$.
 Recall the definition of the normalizing constant $Z_n$ from \eqref{eq:normconst}. 
 The Gibbs variational principle implies 
	$$\log{Z_n}=\sup_{\QQ}\left(\EE_{\QQ}\left[\frac{1}{2}\bs^{\top} \A_n\bs+\boldc^{\top}\bs\right]-\mathrm{KL}\big(\QQ|\prod_{i=1}^n \mu_i\big)\right),$$
	where the supremum is taken over all probability measures $\QQ$ on $[-1,1]^n$ (see \cite[\red{Equation 5.16 in pg. 135}]{wainwright2008graphical}), and is attained at $\QQ=\PP$. The Mean-Field approximation (see \cite{bishop2006pattern,blei2017variational}) for $\log Z_n$ arises by restricting the above supremum to the family of product measures, i.e.,
	\begin{align}\label{eq:MF formula}	
		\log{Z_n}&\ge \sup_{\QQ=\prod_{i=1}^n \QQ_i}\left(\EE_{\QQ}\left[\frac{1}{2}\bs^{\top}\A_n\bs+\boldc^{\top}\bs\right]-\mathrm{KL}\big(\QQ|\prod_{i=1}^n \mu_i\big)\right)=\sup_{\mathbf{z}\in [-1,1]^n} F(\mathbf{z}),
	\end{align}
	where
\begin{align}\label{eq:defF}	
 F(\mathbf{z}):=\frac{1}{2} \mathbf{z}^{\top} \A_n \mathbf{z} + \mathbf{z}^{\top} \bc - \sumin I_i(z_i).
\end{align}
	The equality \eqref{eq:MF formula} follows from simplifying the supremum on the middle term by optimizing over product measures $\QQ = \prod_{i=1}^n \QQ_i$ where each $\QQ_i$ has mean $z_i$; see e.g. Proposition 1 in \cite{yan2020nonlinear}. 
    The following lemma demonstrates that $F(\cdot)$ has a unique maximizer under a high-temperature assumption. 
	
	\begin{lemma}\label{lem:uniqueness of optimizers}
		Assume that the WHT condition \eqref{eq:wht} holds. Then the optimization problem 
		$$\sup_{\mathbf{z}\in [-1,1]^n} F(\mathbf{z})$$
		has a unique global maximizer $\bu \in (-1, 1)^n$. The maximizer $\bu$ satisfies the following fixed point equations: 
		\begin{align}\label{eq:fpeq}
			u_i &= \psi_i'( s_i + c_i),
		\end{align}
		where $s_i := \sumjn\A_n(i,j) u_j$. Additionally, the product measure
  \begin{align}\label{eq:meanfieldmeas}
		\frac{d\pQ}{d\prod_{i=1}^n \mu_i}(\bs) :\propto \prod_{i=1}^n \exp(\phi_i(u_i)\sigma_i)
	\end{align}
corresponds to the maximizer in the first supremum of \eqref{eq:MF formula}, as the mean of the $i$th marginal of $\pQ$ is exactly $u_i$.
	\end{lemma}

\noindent Here, $\bu$ (and hence $\mathbf{s}$) depends on $(\A_n, \bc,\{\mu_i\}_{1\le i\le n})$, but we suppress this dependence for the sake of simplicity. 

The Mean-Field approximation $\pQ = \prod_{i=1}^n \pQ_i$ motivates the natural centered statistic 
 \begin{equation}\label{eq:centered1}
 T_n^*(\bs) := T_n-\sum_{i=1}^n q_i\EE_{\pQ_i}[\sigma_i] = \sum_{i=1}^n q_i(\sigma_i-u_i).
 \end{equation}
 
 \noindent If the inequality \eqref{eq:MF formula} is ``tight enough",  one may expect $T_n^*(\bs)$ to converge to a centered normal distribution. We will prove this under the following condition.
 
	\begin{assume}[Strong Mean-Field regime (SMF) assumption]\label{assn:mf}
   \begin{align}\label{eq:smf}
        \alpha_n:=\max_{i=1}^n \sumjn \A_n^2(i,j)= o(n^{-1/2}).
   \end{align}
	\end{assume}

\noindent 
We refer to the above condition as \emph{Strong Mean-Field assumption}, \red{as this implies the conclusion $ \lVert \A_n\rVert_F^2=o(\sqrt{n})$, which is stronger than the more commonly used Mean-Field assumption $\lVert \A_n\rVert_F^2=o(n)$ (see \cite{basak2017universality, yan2020nonlinear, mukherjee2022variational,Deb2024detecting,deb2020fluctuations}).} We motivate our stronger assumption by presenting our first theorem.

\subsection{Weak laws and concentration inequalities}\label{sec:meanfieldconc} 
The first main result of this Section establishes polynomial and exponential tail bounds for $T_n^*(\bs)$, when $\bs\sim\PP$.
    
    \begin{thm}\label{thm:LLN}
Suppose $\bs\sim\PP$. Assume that the MHT condition \eqref{eq:mht} holds. Then
  we have:
\begin{align}\label{eq:firstbd}		
  \EE\left[T_n^*(\bs)\right]^2 \lesssim \max\big(1, n\alpha_n^2\big).
\end{align}
        Further, if the SHT condition \eqref{eq:sht} holds, then 
        there exist positive constants $h_1$ and $h_2$, such that given any $t>0$, we have:
\begin{align}\label{eq:secondbd}
  \PP\big(\big|T_n^*(\bs)\big|\ge t\big)\le h_1\exp(-h_2 t^2)+h_1 \exp\left(-\frac{h_2  t}{\sqrt{n}\an}\right).
\end{align}
Here, all implicit constants only depend on $\rho$ (see \eqref{eq:mht}).
	\end{thm}

        \cref{thm:LLN} has some important consequences. First, under the SMF condition \eqref{eq:smf} ($\alpha_n\ll n^{-1/2})$, \cref{thm:LLN} immediately implies that 
        $$ T_n^*(\bs) = O_P(1), \quad \mbox{where} \quad \bs\sim \PP.$$ This is what one would expect if $\PP$ were a product distribution with mean vector ${\bu}$, so this bound is the best possible, and one cannot expect a better result even under independence. Secondly,~\cref{thm:LLN} gives non-trivial conclusions even without the SMF condition~\eqref{eq:smf}. To see this, take ${\bf q}=n^{-1/2}{\bf 1}_n$. The law of large numbers $$\overline{\bs}-\overline{\bu}:=n^{-1}\sumin \sigma_i-n^{-1}\sumin u_i=o_P(1)$$ holds as soon as $\alpha_n\to 0$. It is well known that Mean-Field centering can fail if $\alpha_n$ stays bounded away from $0$ (see e.g., \cite{qiu2023tap,celentano2023mean}), so the requirement $\alpha_n\to 0$ in the above display is tight.
       
 \subsection{Central Limit Theorems}\label{sec:clt}
 
 Following the law of large numbers studied in~\cref{thm:LLN}, it is natural to wonder about the non-degenerate fluctuations of $T_n^*(\bs)$ (see \eqref{eq:centered1}) where $\bs\sim \PP$.  This is formalized in the following result, which provides a finite-sample Berry-Esseen type bound for the Kolmogorov–Smirnov distance (see \eqref{eq:ksdist}) between the law of $T_n^*(\bs)$ and a centered Gaussian distribution. 
\begin{thm}\label{thm:clt}
    Suppose that $\bs\sim \PP$ and the MHT condition \eqref{eq:mht} holds. Let $\bu$ be the Mean-Field optimizers as in  \cref{lem:uniqueness of optimizers}. 
    For all $i= 1,\ldots,n$, let $\nu_i := \sumjn \A_n(i,j) \psi_j'(c_j)$.
   Assume $\upsilon_n:=\sumin q_i^2 \psi_i''(c_i) \red{\ge \kappa}$ for some constant \red{$\kappa \in (0,1)$}, and define the following:
    \begin{equation}\label{eq:random field}
       R_{1n} := \sumin \left(\sumjn \A_n(i,j) q_j (\psi_j''(c_j) - \upsilon_n)\right)^2,\qquad
        R_{2n} := \sumin \nu_i^2,\qquad   R_{3n}:=\sumin \nu_i^4.
    \end{equation}
   Fix $\lambda_n$ such that $\red{|\lambda_n| \le \rho}$ (see \eqref{eq:mht}), and set $\boldsymbol{\epsilon} := \A_n \bq - \lambda_n \bq$. For $W_n \sim \mcn \left(0, \frac{\upsilon_n}{1 - \lambda_n \upsilon_n} \right)$, we have
    \begin{align}\label{eq:ksbd}
    &\;\;\;d_{KS}\left(T_n^*(\bs) \, ,\, W_n\right)\lesssim \sqrt{R_{1n}} + \sqrt{\an R_{2n}} + \sqrt{R_{3n}} + \sqrt{n} \an + \lVert \bq\rVert_{\infty} + \lVert \boldsymbol{\epsilon} \rVert .
    \end{align}
    Here the implicit constant only depends on $\rho$ and $\kappa$.
\end{thm}

\begin{remark}[Eigenvalue-eigenvector]\label{rem:eigen}
    The upper bound in \cref{thm:clt} is trivial unless  $\| \boldsymbol{\epsilon} \| =\lVert \A_n \bq - \lambda_n \bq\rVert\to 0$. In other words, we need $(\lambda_n,\bq)$ to be an approximate eigenvalue-eigenvector pair. Note that, if $\lVert \A_n \bq - \lambda_n \bq\rVert\to 0$, then 
    $$\limsup\limits_{n\to\infty}|\lambda_n|\le \limsup\limits_{n\to\infty}\,\lVert \A_n\bq\rVert\le \limsup\limits_{n\to\infty}\,\lVert \A_n\rVert_4 \le  \rho < 1,$$
    where we have used \cref{lem:normorder} coupled with the MHT condition \eqref{eq:mht}. Further, $$ \upsilon_n= \sumin q_i^2 \psi_i''(c_i)\le 1.$$
    Hence, the Normal distribution $W_n$ is well-defined for all $n$ large enough, as its variance is bounded away from $\infty$. 
    \end{remark}

\begin{remark}[Delocalization]\label{rem:eigen2}
    For $T_n^*(\bs)$ to be asymptotically normal, we also need $\bq$ to be delocalized in the sense that $\lVert \bq\rVert_{\infty}=\lVert \bq\rVert_{\infty}/\lVert \bq \rVert\to 0$, i.e., a single term cannot dominate the summand in $T_n^*(\bs)$. This is reminiscent of the Uniform Asymptotic Negligibility (UAN) condition in the context of the usual central limit theorem with independent observations.
\end{remark}

\begin{remark}[Necessity of  \eqref{eq:smf}]\label{rmk:necessity of smf}
    Note that, for $T_n^*(\bs)$ to converge weakly to a mean $0$ Gaussian distribution, the SMF condition \eqref{eq:smf} is required. If $\sqrt{n}\alpha_n$ converges to a positive real instead, then the limiting Normal may not be centered at $0$ for specific choices of $\A_n$ and $\bc$; see e.g. \cite[Example 1.3]{deb2020fluctuations}. In this sense, the SMF condition is tight.
\end{remark}

   For both Theorems \ref{thm:LLN} and \ref{thm:clt}, we have centered $T_n(\bs)$ by subtracting $\bq^{\top}\bu$, where the Mean-Field solution $\bu$ is obtained from \cref{lem:uniqueness of optimizers}. Typically, ${\bu}$ needs to be computed using numerical optimization methods. The purpose of the following result is to show that, at the expense of some additional error terms, $\sumin q_i\sigma_i$ has Gaussian fluctuations under $\PP$, around a more explicit function of $\bq$, $\A_n$, and $\bc$. 
   
\begin{thm}\label{thm:CLT for u}
    Suppose $\bs\sim \PP$ and consider the same notation as in \cref{thm:clt}. 
    Suppose that the MHT condition \eqref{eq:mht} holds and assume $|\lambda_n| \le \rho, \upsilon_n \ge \kappa > 0$. Recall that $\boldsymbol{\epsilon} = \A_n \bq - \lambda_n \bq$. Set $\Psi'(\bc) := \left(\psi_1'(c_1), \ldots, \psi_n'(c_n) \right)^\top$ and 
    $$R_{4n} := \left|\sum_{1 \le i,j \le n} \A_n(i,j) q_i(\psi_i''(c_i) - \upsilon_n) \psi_j'(c_j) \right|.$$
    \begin{enumerate}
        \item[(a)]
Then we have
    \begin{align}\label{eq:CLT for u}
         \bigg|\bq^\top \left(\bu - \frac{\Psi'(\bc)}{1 - \lambda_n \upsilon_n} \right)\bigg| \lesssim \sqrt{R_{1n} R_{2n}} + \sqrt{R_{3n}} + R_{4n} + \sqrt{R_{2n}}\lVert \boldsymbol{\epsilon} \rVert + \big|\boldsymbol{\epsilon}^{\top}\Psi'(\bc)\big|.
    \end{align}

 \item[(b)]  Further we have
    \begin{align}
    &\;\;\;\;d_{KS}\left(\sumin q_i\left(\sigma_i-\frac{\psi_i'(c_i)}{1-\lambda_n\upsilon_n}\right)\, , \, W_n\right)\nonumber\\ &\lesssim \sqrt{R_{2n}}\left(\sqrt{R_{1n}} + \sqrt{\alpha_n} + \lVert \boldsymbol{\epsilon}\rVert\right) + \sqrt{R_{1n}} +  \sqrt{R_{3n}}  + \sqrt{n} \an + \lVert \bq\rVert_{\infty}+ R_{4n} +\big|\boldsymbol{\epsilon}^{\top}\Psi'(\bc)\big| + \red{\|\boldsymbol{\epsilon}\|}.\label{eq:erbdcltforu}
    \end{align}
    In both cases, the implicit constant depends only on $\rho$ from \eqref{eq:mht} and $\kappa$.
    \end{enumerate}
\end{thm}

In \cref{thm:CLT for u}, neither the centering term nor the variance of $W_n$ requires the computation of the Mean-Field solution $\bu$. This is the main benefit of~\cref{thm:CLT for u} when compared to~\cref{thm:clt}.

\subsection{Proof overview and main challenges}

This paper studies concentration and CLT for linear statistics in a Gibbs measure with a quadratic Hamiltonian. The centering for both the concentration and CLT is provided by the vector ${\bf u}$, which is the expectation of the Mean-Field variational approximation to this Gibbs measure. Under a suitable high temperature assumption on the interaction matrix of the quadratic statistic, the mean field optimizer is unique. However, as opposed to a Dobrusin type high temperature assumption which involves the $\ell_\infty$ operator norm, this paper works with the $\ell_4$ operator norm, which is much less restrictive. In particular, this allows us to verify a CLT for a broader class of asymptotic regimes, as demonstrated in Example \ref{ex:Hopfield clt}. We stress that our techniques allow for both ferromagnetic and anti-ferromagnetic models, as well as some spin glass models, as demonstrated via examples.
\\

Our first main result is a concentration for linear statistics (see \cref{thm:LLN}), centered via the Mean-Field centering. The starting point of our proof is a concentration for linear statistics with a centering via conditional expectation $$\EE[\sigma_i|(\sigma_j, j\ne i)]=\psi'(m_i+c_i),$$
 which we derive  using the method of exchangeable pairs. Here $m_i=\sum_{j=1}^n\A_n(i,j)\sigma_j$ is the local field at $i$, and we assume common base measures $\mu_i=\mu$ to simplify the notation $\psi_i = \psi$. The main challenge is to derive a concentration with respect to the Mean-Field centering, given by an implicit vector ${\bf u}$ which satisfies the fixed point equation $$u_i=\psi'(s_i+c_i),\qquad \text{ where }\qquad s_i=\sum_{j=1}^n\A_n(i,j)u_j.$$ In order to get a concentration with respect to ${\bf u}$ we can write
\begin{align*}
\sum_{i=1}^nq_i(\sigma_i-u_i)=\sum_{i=1}^nq_i(\sigma_i-\psi'(m_i+c_i))+\sum_{i=1}^nq_i(\psi'(m_i+c_i)-u_i).
\end{align*}
The first term in the RHS above can be controlled using the conditionally centered concentration. Proceeding to bound the second term, a Taylor's series expansion gives
\begin{align*}
     \sumin q_i(&\psi'(m_i+c_i) - u_i) = \sumin q_i(\psi'(m_i+c_i) - \psi'(s_i+c_i)) \\
     &= \sumin q_i (m_i - s_i) \psi''(s_i + c_i) + \red{\frac{1}{2}} \sumin q_i (m_i - s_i)^2 \psi'''(\xi_i + c_i) \\
     &= \sumij \A_n(i,j) q_i (\sigma_j - \red{u_j}) \psi''(s_i + c_i) + O\Big(\sumin |q_i| (m_i - s_i)^2 \Big) \\
     &= \sumjn \tilde{q}_j (\sigma_j - \red{u_j}) + O\Big(\sqrt{\sumin  (m_i - s_i)^4} \Big),
\end{align*}
where $$\tilde{q}_j :=\sum_{\red{i=1}}^n\C_n(i,j)q_i,\qquad \text{ with }\qquad \C_n(i,j)= \A_n(i,j) \psi''(s_i+c_i).$$ %
Thus we have reduced the problem of concentration of ${\bf q}^\top \bs$ to that of ${\bf q}^\top \C_n\bs$. In this way, a recursive argument allows us to focus on the behavior of ${\bf q}^\top \C_n^\ell \bs$ for some large positive integer $\ell$. Under the high temperature assumption we have that above sequence of random variables converge to $0$ geometrically fast, thereby allowing us to derive the desired concentration about $\bu$. Controlling the error term $\sum_{i=1}^n(m_i-s_i)^4$ necessitates the high temperature assumption $\|\A_n\|_4\le \rho<1$ via the fourth operator norm. We note here that if we had additional structure on ${\bf q}$ (say $\|{\bf q}\|_\infty\lesssim n^{-1/2})$, then the error bound could have been controlled by the weaker high temperature assumption $\|\A_n\|\le \rho<1$. Instead we focus to avoid extra assumptions on the eigenvector, and work with the fourth operator norm. 
\\

Next, the proof \cref{thm:clt} is based on Stein's method for exchangeable pairs. For simplicity, we sketch the proof assuming that $\A_n \bq = \lambda_n \bq$. To elaborate, we first construct an exchangable pair $(\bs, \bs')$ by the usual way:

Let $I$ be a randomly sampled index from $\{1,2,\ldots ,n\}$. Given $I=i$, replace $\sigma_i$ with an independently generated $\sigma'_i$ drawn from the conditional distribution of $\sigma_i$ given $\sigma_j,\ j\neq i$. Then, setting
$$T_n^*(\bs')=\bq^{\top}(\bs'-\bu), \quad \text{we have}\quad T_n^*(\bs)-T_n^*(\bs')=q_I(\sigma_I-\sigma_I').$$ 
Noting that $\EE[\sigma_{i}' \mid \bolds] = \psi'( m_i + c_i)$, a one term Taylor's series gives
	\begin{align*}
		\EE (T_n^*(\bs) - T_n^*(\bs') \mid \bs) &= \frac{1}{n} \sumin q_i(\sigma_i - \psi'( m_i + c_i)) 
		\\ &\approx \frac{1}{n} \sumin q_i \left(\sigma_i - u_i - (m_i - s_i)\psi''( s_i + c_i) \right)\\
        =& \frac{T_n^*(\bs)}{n} - \frac{1}{n} \sumin q_i \psi''( s_i + c_i) \Big(\sumjn \A_n(i,j)(\sigma_j - u_j)\Big) 
        \end{align*}
        which is a good approximation, if $m_i\approx s_i$ (this is argued in \cref{lem:m-n contraction}).
        Provided $s_i\approx 0$ (which we argue in \cref{lem:optimizer u l2 concentration}), the RHS above can be approximated as
        \begin{align*}
	 \frac{T_n^*(\bs)}{n} - \frac{1}{n} \sumjn \Big(\sumin \A_n(i,j) q_i \psi''(c_i)\Big) (\sigma_j - u_j) \approx\frac{(1 -  \lambda_n \upsilon_n) T_n^*(\bs)}{n}.
	\end{align*}
	Here, the last line follows on noting that for $(c_1,\cdots,c_n)$ i.i.d., we have the concentration
    \begin{align}\label{eq:sketch pf approximation}
        \sumin  \A_n(i,j) q_i \psi''(c_i) \approx \sumin \A_n(i,j) q_i \EE \psi''(c_1) = \lambda_n q_j \EE \psi''(c_1),
    \end{align}
    where the last equality uses the eigenpair assumption $\A_n \bq = \lambda_n \bq$. As a proxy for $\EE \psi''(c_1)$, we use the quantity $\upsilon_n=\sum_{i=1}^nq_i^2 \psi''(c_i)$, which turns out to be more efficient for controlling associated error terms. Now, the above implies that the Stein's equation holds approximately. The final KS bound in the RHS of \eqref{eq:ksbd} results from the various approximation errors in the above expansion.
    \\

Finally, \cref{thm:CLT for u}, which gives a more tractable form for the centering, follows from  Taylor expansion of the Mean-Field optimizers $\bu$ and simplifying the expression for $\bq^\top \bu$ using the approximate eigenpair condition. To begin, use a one term
Taylor expansion to get
    \begin{align*}
        u_i = \psi'(s_i+c_i) \approx \psi'(c_i) +  s_i \psi''(c_i),
    \end{align*}
    which gives $\bq^\top \bu$ as
    \begin{align*}
        \bq^\top \bu &\approx \bq^\top \Psi'(\bc) + \sumin s_i q_i \psi''(c_i) = \bq^\top \Psi'(\bc) +  \sumij\A_n(i,j) u_j q_i \psi''(c_i) \approx \bq^\top \Psi'(\bc) + \upsilon_n \lambda_n \sumjn u_j  q_j,
    \end{align*}
    where
  the last line above again uses concentration followed by the eigenpair assumption of \eqref{eq:sketch pf approximation}. 
    Now, the claim $\bq^\top \bu \approx \frac{\bq^\top \Psi'(\bc)}{1 -  \lambda_n \upsilon_n}$
    follows by adjusting the terms.
\\

Even though our main application is the random field Ising model with i.i.d. magnetic fields, our main results are stated for general inhomogeneous fields, deterministic or random. Allowing this broad generality comes at a cost, as now it is difficult to pinpoint the geometry of the space of Mean-Field optimizers $\bu$. The MHT assumption comes to the rescue, forcing the uniqueness of optimizer, and also giving concentration for various statistics, which are helpful to verify the CLT, but also of possible independent interest. We believe that using the tools developed in this paper, one can study concentration and CLT of linear statistics under a broader class off magnetic fields ${\bf c}$ (such as dependent coordinates, which comes up while analyzing the posterior of a high dimensional linear regression model).

\subsection{Random field Ising models}\label{sec:isingrandom}

  Recall the RFIM introduced in \cref{ex:rfim}. In this section we assume that the base measures $\mu_i=\mu$ for all $i\in [n]$, where the common measure $\mu$ is  supported on $[-1,1]$. Consequently, we omit the subscript $i$ in the $\psi_i$s (see \cref{def:expfam}).
  We also assume that $\bc = (c_1, \ldots, c_n)^\top$ are i.i.d. random variables from a distribution $F$, where $F$ is a distribution on reals. Throughout this section \red{let $\EE_F$ denote the expectation with respect to the random field $\bc$}, and assume that $\EE_F \psi'(c_1)=0$. In particular this assumption holds if both the measures $\mu$ and $F$ are symmetric. Thus, this can be viewed as a generalization of the constant field Ising model where $F$ is degenerate at $0$ (see \cite{Deb2024detecting,Reza2018}).  
  Recall the definitions of $\alpha_n$ and $\psi(\cdot)$ from \eqref{eq:rowcontrol} and \cref{def:expfam} respectively.
  In our next result, we provide \emph{both quenched (conditional on $\bc$) and annealed (unconditional) central limit theorems} for the linear statistic $T_n$ after appropriate centering.

    \begin{thm}\label{cor:i.i.d. clt}
    Assume ${\bf A}_n$ satisfies the MHT condition \eqref{eq:mht}. Set $\upsilon := \EE_F \psi''(c_1) > 0$, and define 
    $$\textsc{Err}(\A_n, \bq) := \sqrt{\an} + \sqrt{n} \an + \|\bm \epsilon\| + \|\bq\|_\infty.$$
    \begin{enumerate}[(a)]
        \item Then we have
        \begin{align*}\
            d_{KS}\Big(\sumin q_i(\sigma_i - u_i), W_\infty\mid \bc\Big) = O_P\Big( \textsc{Err}(\A_n, \bq) \Big),\quad \text{where} \quad W_\infty\sim \mcn\left(0, \frac{\upsilon}{1 -  \lambda \upsilon }  \right).
        \end{align*}
        
        \item Also we have
        \begin{align}
            d_{KS}\left(\sumin q_i\left(\sigma_i - \frac{\psi'(c_i)}{1-\lambda \upsilon}\right), W_\infty \mid \bc\right)&= O_P\Big(\textsc{Err}(\A_n, \bq) + \sqrt{n \an} \|\bm \epsilon\|\Big),\notag \\
           d_{KS}\left(\sumin q_i \sigma_i, W_\infty+\widetilde{W}_\infty\right) &= O\Big(\textsc{Err}(\A_n, \bq) + \sqrt{n \an} \|\bm \epsilon\|\Big),\label{eq:rfim clt centered}
        \end{align}
        where $W_\infty$ and $\widetilde{W}_\infty\sim \mcn\left(0,  \frac{\EE_{F} [\psi'(c_1)^2]}{(1 - \lambda \upsilon)^2} \right) $ are independent. 
    \end{enumerate}
  For both parts (a) and (b), the implicit constant depends only on $\rho$ from \eqref{eq:mht} and $\upsilon$.
    \end{thm}
In particular, for the usual Ising model with binary spins, one takes $\mu$ to be supported on $\pm 1$ with $\mu(\pm 1) = \frac{1}{2}$, which gives $\psi'(c) = \tanh(c)$ and $\psi''(c)= \sech^2(c)$. 
    The annealed limiting normal distribution in \eqref{eq:rfim clt centered} (without any centering) might, at first glance, seem unfamiliar given the nature of our main theorems earlier. However, note that, due to the independence of  $\{c_i\}_{i\in  [n]}$, the sum $\sum_{i=1} q_i \psi_i'(c_i)$ converges to a centered normal distribution by a standard Lindeberg-Feller argument. Combining this observation with \cref{thm:CLT for u} yields the required annealed limit. This also highlights the benefits of refining the centering in \cref{thm:CLT for u} relative to \cref{thm:clt}, albeit under slightly stronger assumptions.

    \begin{remark}[Eigenvalue-eigenvector]
        \red{Continuing from Remarks \ref{rem:eigen} and \ref{rem:eigen2}, we need an approximate eigenvalue $\lambda$ (free of $n$) and a sequence of approximately delocalized eigenvectors $\bq \in \mathbb{R}^n$, i.e.
  \begin{equation}\label{eq:standassn}
  \lVert \bq \rVert = 1,\quad \lVert \A_n \bq - \lambda \bq \rVert \to 0, \quad \lVert \bq \rVert_{\infty} \to 0
  \end{equation}
  to get a CLT (i.e., for the bounds in \cref{cor:i.i.d. clt} to converge to 0).
  In part (a), the RHS is $o_P(1)$ when \eqref{eq:standassn} holds and $\alpha_n \ll \frac{1}{\sqrt{n}}$. Similarly, in part (b), the RHS is $o_P(1)$ and $o(1)$ respectively, under \eqref{eq:standassn}, $\alpha_n \ll \frac{1}{\sqrt{n}}$, and $\lVert \bm\epsilon \rVert = o\left(\frac{1}{\sqrt{n \an}}\right)$.
  We will provide specific choices of $\lambda$ and $\bq$ that satisfy \eqref{eq:standassn} in the examples that follow. }
    \end{remark}
    
    \begin{remark}[Tightness of the Berry-Esseen bound]             
        To see the tightness of our bounds, take  $\bq = n^{-1/2} \mathbf{1}$ and $\A_n$ to be the adjacency matrix of a complete  graph scaled by its degree $n-1$. In other words, consider the sample mean under the Curie-Weiss model. Then, we have $\alpha_n \lesssim \frac{1}{n}$. Consequently, it follows from \cref{cor:i.i.d. clt} that
        $$d_{KS}\left( n^{-1/2}\sumin (\sigma_i-u_i)\,,\, W_\infty \mid \bc \right) = O_P(n^{-1/2})\, ,\quad d_{KS}\left(\sumin q_i \sigma_i\,,\, W_\infty+\widetilde{W}_\infty \right) = O(n^{-1/2}).$$
        The $n^{-1/2}$ rate matches the well-known Berry-Esseen rate for the sample mean of i.i.d. random variables \citep{esseen1942liapunov}, so our rates are tight in this example. Also, if $\alpha_n\asymp n^{-1/2}$, it is known that the above CLT fails for the choice $c_i=c$ for all $i\in [n]$ as mentioned in \cref{rmk:necessity of smf}.
        In general, the rate gets slower as $\alpha_n$ and $\lVert \bq \rVert_\infty$ increase.
    \end{remark}
  
Now, we provide concrete choices of the coupling matrix $\A_n$ including (1) adjacency matrices for various simple graphs, (2) the Hopfield model. For each example, we show how \cref{cor:i.i.d. clt} can be applied to derive limiting distributions of $T_n$. When the coupling matrix $\A_n$ is random (for example, for random graphs or the Hopfield model), the following statements will always condition on the randomness of $\A_n$.

\begin{ex}[Erd\H{o}s-R\'{e}nyi graphs]
    Let $\G_n \sim ER(n,p_n)$ be the Erd\H{o}s-Rényi random graph with $n$ vertices and edge probability $p_n$, and let $\A_n := \frac{\theta \G_n}{n p_n}$. Here, $\theta$ is often-called the inverse-temperature parameter in the statistical physics literature.  It is easy to check that both the matrix norms $\lVert\A_n\rVert$ and  $\lVert \A_n\rVert_{\infty}$ converge in probability to $ |\theta|$. Then the following result is an immediate application of \cref{cor:i.i.d. clt}, where all results are conditional on the random graph $\G_n$.
\end{ex}
    
    \begin{cor}\label{lem:ER}
  Let $\bolds$ be an observation from the model~\eqref{eq:model}, where the matrix ${\bf A}_n$ is obtained from the Erd\H{o}s-R\'{e}nyi graph as described above. Assume further that $p_n\gg n^{-1/2}$ and $|\theta|<1$. 

    \vspace{0.02in}
    
    \noindent (a). \emph{(Sample mean)}. Then we have
    \begin{eqnarray*}
    \frac{1}{\sqrt{n}}\sum_{i=1}^n (\sigma_i-u_i)\mid \bc &\overset{d}{\longrightarrow}& \mcn\left(0,\frac{\EE_F \psi''(c_1)}{1-\theta \EE_F \psi''(c_1)}\right), \\
    \frac{1}{\sqrt{n}}\sum_{i=1}^n \sigma_i &\overset{d}{\longrightarrow}& \mcn\left(0,\frac{\EE_F \psi''(c_1)}{1-\theta \EE_F \psi''(c_1)}+\frac{\EE_F [\psi'(c_1)^2]}{(1-\theta \EE_F \psi''(c_1) )^2}\right),
    \end{eqnarray*}

    \vspace{0.02in}

    \noindent (b). \emph{(Contrasts)}. If $\bq$ satisfies
\begin{equation}\label{eq:contrastcon}
    \lVert \bq\rVert=1, \qquad \sum_{i=1}^n q_i = o(\sqrt{n p_n}), \qquad \lVert \bq \rVert_{\infty} \to 0,
    \end{equation}
then we have
    \begin{eqnarray*}
    \sum_{i=1}^n q_i(\sigma_i-u_i)\mid \bc &\overset{d}{\longrightarrow} &\mcn\left(0,\EE_F \psi''(c_1)\right), \\ \sum_{i=1}^n q_i \sigma_i &\overset{d}{\longrightarrow} &\mcn\left(0,\EE_F \psi''(c_1)+\EE_F [\psi'(c_1)^2]\right).
    \end{eqnarray*}
    \end{cor}
Note that the regime $|\theta|<1$ covered in the above result allows for both the ferromagnetic regime $\theta\in (0,1)$ and the anti-ferromagnetic regime $\theta\in (-1,0)$.

\begin{remark}[Universal limit for contrast vectors]
    An interesting feature of \cref{lem:ER}(b) is that if ${\bm q}$ satisfies \eqref{eq:contrastcon} (i.e.~it is approximately a contrast), the limit distribution is universal, i.e., it does not depend on $\bq$. This is because, for all $\bq$ satisfying \eqref{eq:contrastcon}, $(\bq,0)$ forms an approximate eigenvector-eigenvalue pair in the sense of \cref{cor:i.i.d. clt}. We will now use this observation to show joint convergence for multiple contrasts.
\end{remark}
\begin{defi}
For two functions $h_1,h_2\in L^2\big((0,1]\big)$, define their inner product as
$\langle h_1,h_2\rangle:=\int_0^1 h_1(x)h_2(x)dx$.
\\

    \noindent For any vector ${\bf x}=(x_1,\cdots,x_n)$, define a piecewise constant function $g_{\bf x}:(0,1]\mapsto \R$ given by
    $$g_{\bf x}(t)=x_i\qquad\text{ if }\qquad\frac{i-1}{n}<t\le \frac{i}{n}, 1\le i\le n.$$
\end{defi}

\begin{cor}[Multivariate convergence for contrasts]\label{cor:ER multidimensional}
    Suppose we are in the set up of \cref{lem:ER}. %
    Let
    $$\mathcal{Q} \in \red{L^2\big((0,1]\big)}\quad \text{s.t.} \quad \int_{0}^1 Q(x) dx = 0, \quad \int_{0}^1 Q(x)^2 dx = 1\Big\}$$ %
    be the collection of contrast functions. 
   Let $S_\infty, \tilde{S}_\infty:\mathcal{Q} \to \mathbb{R}$ be Gaussian processes with mean 0 and covariance functions $$K(Q_1, Q_2) := [\EE_F \psi''(c_1)] \langle Q_1, Q_2\rangle, \quad \tilde{K}(Q_1, Q_2) := \Big[\EE_F \psi''(c_1) + [\EE_F \psi'(c_1)]^2\Big] \langle Q_1, Q_2\rangle,$$ respectively. 
    
    \begin{enumerate}
        \item[(a)]Then, conditional on $\bc$, the finite dimensional marginals of the process $\Big\{\sqrt{n}\langle Q,g_{\bs}-g_{\bf u}\rangle,Q\in \mathcal{Q}\Big\}$ converge to the finite dimensional marginals of the Gaussian process $\{S_\infty(Q),Q\in \mathcal{Q}\}$.

    \item[(b)]Unconditionally, the finite dimensional marginals of the process $\{\sqrt{n}\langle Q, g_{\bs} \rangle,Q\in \mathcal{Q}\}$ 
    converge to finite dimensional marginals of the Gaussian process $\{\tilde{S}_\infty(Q),Q\in \mathcal{Q}\}$.

    \end{enumerate}
\end{cor}

\begin{ex}[$d(n)$-regular graph]
    Let $\G_n$ denote the adjacency matrix of a $d(n)$-regular graph \emph{(deterministic or random)}, and set $\A_n:=\frac{\theta \G_n}{d(n)}$ for some $\theta\in\R$.  It is easy to check that $\lVert\A_n\rVert = \lVert \A_n\rVert_4 = \lVert \A_n\rVert_{\infty} = |\theta|$. 
Again the following result is an immediate application of \cref{cor:i.i.d. clt}.
    \begin{cor}\label{cor:reg}
    Let $\bolds$ be an observation from the model~\eqref{eq:model}, where the matrix $\A_n$ is obtained from a $d(n)$-regular graph as described above. Assume further that $d(n)\gg \red{n^{1/2}}$ and $|\theta|<1$. 

    \vspace{0.02in}

    \noindent\emph{(a).} (Sample mean). Then both conclusions of part (a) in \cref{lem:ER} hold.

    \vspace{0.02in}

    \noindent\emph{(b).} (Contrasts). Suppose further that \eqref{eq:contrastcon} holds with $np_n$ replaced by $d(n)$, and
    $\A_n$ satisfies
    \begin{align}\label{eq:expander}\lVert \A_n - n^{-1}\mathbf{1}\mathbf{1}^{\top}\rVert \lesssim d(n)^{-1/2}.\end{align}
    Then both conclusions of part (b) in \cref{lem:ER} hold. %
    \end{cor}
\end{ex}
We note that \eqref{eq:expander} is an expander type condition, which is satisfied by several regular graph ensembles of common interest, such as
random regular graphs and Ramanujan graphs. Also, by repeating the argument in \cref{cor:ER multidimensional}, one can get joint convergence of multiple contrast vectors.
\\

In the above two examples we consider graphs/matrices which are approximately regular. Our next result shows that we can handle irregular graphs as well.

\begin{ex}[$f$-random graphs \cite{borgs2018L2,borgs2019L1}]
    Let $f:[0,1]^2 \to [0,1]$ be a symmetric bounded measurable function, and let $Q:[0,1] \to \R$ be a normalized eigenfunction with eigenvalue $\lambda$, i.e. $$\int_0^1 Q(y)^2 dy = 1,\qquad
\int_0^1 f(x,y) Q(y) dy = \lambda Q(x), \quad \forall x \in [0,1].$$ %
    For $\red{\mathbf{U}} := (U_1, \ldots, U_n)^\top \stackrel{i.i.d}{\sim} \textnormal{Unif}[0,1]$ and $0 \le \gamma < 1/2$, consider an adjacency matrix
    $$\G_n(i,j) \stackrel{ind}{\sim} \textnormal{Ber}\left(\frac{f(U_i, U_j)}{n^\gamma}\right), \quad \forall i < j,$$
    and define $\A_n := \frac{\theta \G_n}{n^{1-\gamma}}$. Note that the random graph $\G_n$ is dense if $\gamma = 0$, and the above scaling allows sparser graphs as well. Also, note that $f$ is not necessarily continuous, thus allowing graphs with block structures such as stochastic block models or complete bipartite graphs.
    The following corollary shows that our CLTs still hold for such cases.
\end{ex}

\begin{cor}\label{cor:graphon}
Let $\bolds$ be an observation from the model~\eqref{eq:model}, where the matrix $\A_n$ is the $f$-random graph from the above example. Assume further that $\EE f(U_1, U_2) > 0$ and $|\theta|<1$.
Then, the following holds: 
\begin{align*}
    \frac{1}{\sqrt{n}}\sumin Q(U_i)(\sigma_i - u_i) \mid \bc &\xd \mcn\Big(0, \frac{\EE_F \psi''(c_1)}{1- \theta\lambda \EE_F \psi''(c_1)}\Big) \\
    \frac{1}{\sqrt{n}}\sumin Q(U_i) \sigma_i &\xd \mcn\Big(0, \frac{\EE_F \psi''(c_1)}{1-\theta\lambda \EE_F \psi''(c_1)} + \frac{\EE_F[\psi'(c_1)^2]}{(1-\theta\lambda \EE_F \psi''(c_1))^2} \Big).
\end{align*}
\end{cor}

In our final example, we give a spin-glass model, where the entries of $\A_n$ are both positive and negative. For simplicity, we only consider the case when the base measure $\mu_i=\mu$ puts equal mass on $\pm 1$.

\begin{ex}[Diluted Hopfield model/covariance matrix]\label{ex:Hopfield clt}
Suppose the base measure $\mu_i\equiv \mu$ in \eqref{eq:model} is the Rademacher distribution. Let $\{Z_{k,i}\}_{k\in [N],i\in [n]}$ be independent mean $0$ and variance $1$  sub-Gaussian random variables. Define a symmetric random matrix ${\bf M}_n$ %
by setting $${\bf M}_n(i,j):=\frac{1}{N}\sum_{k=1}^N Z_{k,i}Z_{k,j}.$$ %
Let $\G_n$ be the $n \times n$ adjacency matrix of a simple graph (with $\G_n(i,i) = 0$), and define the coupling matrix $\A_n$ as the Hadamard product of $\theta\M_n$ and $\G_n$, for some inverse-temperature parameter $\theta \in \R$:
$$\A_n(i,j) := \theta \M_n(i,j) \G_n(i,j).$$
This is the so-called diluted Hopfield model \citep{bovier1992rigorous, bovier1993rigorous} that considers the Hopfield model on graphs. Note that taking $\G_n$ as the adjacency matrix of a complete graph results in usual Hopfield model without any dilution \citep{hopfield1982neural}, with $\A_n(i,j) = \theta \M_n(i,j)$ for $i \neq j$. %
\end{ex}

\begin{cor}\label{lem:hopfield condition checking}
Let $\bs$ be an observation from the model \eqref{eq:model}, where the matrix ${\bf A}_n$ is coming from the above diluted Hopfield model, with $\mu_i$ denoting the symmetric Rademacher distribution. Let $\bq$ be any vector such that $ \lVert \bq\rVert=1$ and $\lVert \bq \rVert_{\infty}\to 0$. Under the asymptotic scaling $N \gg n^{3/2}$ and any inverse-temperature parameter $\theta$ (not depending on $n$), we have: %
    \begin{align*}
             \sumin q_i\left(\sigma_i - u_i\right) \mid \bc \overset{d}{\longrightarrow} \mcn\left(0, \EE_F \sech^2(c_1) \right), \quad \sumin q_i \sigma_i \overset{d}{\longrightarrow} \mcn\left(0,{1} \right).
        \end{align*}    
\end{cor}

\begin{remark}\label{rmk:hopfield}
    The above corollary applies to \emph{any vector} $\bq$ which is delocalized, in the sense that $\|\bq\|_\infty\to 0$. 
    We stress that without this condition CLT fails even when $\bs$ is a vector of i.i.d. random variables. Thus, essentially there is no restriction on the vector ${\bf q}$.
    We can also allow for dependence across the entries of the rows of the data matrix ${\bf Z}:=(Z_{k,i})_{k\in [N], i\in [n]}$, at the expense of demanding that $\bq$ is an approximate eigenvector of the corresponding population covariance matrix (after removing the diagonal entries).
\end{remark}

Note that we work under an asymptotic scaling where the matrix ${\bf Z}$ has significantly more rows than columns, as opposed to the proportional asymptotics with $n/N \to \gamma \in (0, \infty).$ %
Also, note that we can allow an arbitrary diluting of our Hopfield model. Indeed, this is because we do not re-scale $\A_n$ by the average degree of the graph $\G_n$, in contrast to what is done in \citep{bovier1992rigorous}, which considers the randomly diluted Hopfield model where $\G_n$ is the Erdős–Rényi random graph with connection probability $p$, and divides each $\A_n(i,j)$ by $p$. Our main results directly apply to this version of the diluted Hopfield model from \citep{bovier1992rigorous}, as long as the Erdős–Rényi probability $p$ is bounded away from $0$ (and more generally to dense graphs). It remains unclear whether our results apply to the diluted model for non-dense graphs, scaled by the average degree. The main challenge in this direction is to verify the MHT condition. \\

We end this section on RFIMs by discussing additional potential examples.

\begin{remark}[Broader applications]\label{rem:appeal}
    We have presented Theorems \ref{thm:LLN}, \ref{thm:clt}, and \ref{thm:CLT for u}, that allow general $\A_n$, $\bc$, $\mu_i$s. In particular, these results do not rely on the assumption that $c_1,\ldots ,c_n$ are i.i.d. Thus, we expect them to be applicable to more general quadratic interaction models. A particularly challenging case is when $\PP$ is the posterior distribution in a linear regression model with Gaussian errors. In that case, the corresponding field vector is neither identically distributed nor does it possess independent coordinates, and the base measures $\mu_i$ are heterogeneous. We plan to pursue this problem in the future.
\end{remark}

\subsection{Discussion}\label{sec:conclusion}
In this paper, we have proposed a LLN and CLT for linear statistics of quadratic interaction models. We have demonstrated our results on the random field Ising models. While there has been a lot of recent focus on CLTs in quadratic interaction models
(see e.g. \cite{deb2020fluctuations,kabluchko2019fluctuations,Kabluchko2022} and the references therein), we would like to point out that we work under a unique set of assumptions that has not been considered. The previous literature on Ising models cited above considered (1) binary spins on $\{-1,1\}$, (2) coupling matrices $\A_n$ which are either approximately regular (complete graphs, Erd\H{o}s-R\'{e}nyi random graphs, etc.), or dense graphs, a (3) constant magnetic field, and (4) established limit theorems for the sample mean. In this paper we allow the random variables to take arbitrary bounded values, with a non-constant (potentially random) external field. Additionally, we do not require $\A_n$ to be either regular or dense, and allow for both positive and negative entries. Finally, we establish limit theorems for linear statistics where the coefficients can be arbitrary eigenvectors. Thus, our results go significantly beyond the existing literature.

Our paper opens many interesting directions for future research. One immediate question is can we extend our results to $\bq$ which is not necessarily an approximate eigenvector of the coupling matrix $\A_n$. As our limiting distribution depends on the corresponding eigenvalue $\lambda\equiv\lambda_n$, it is unclear if such a result would be possible at the generality of this paper. Another related question is to consider multivariate fluctuations of simultaneously many eigenvectors, that is when we replace the vector $\bq \in \mathbb{R}^n$ by a matrix $\mathbf{Q} \in \mathbb{R}^{n \times q}$ for some positive integer $q\ge 2$.

A second direction is to improve the high-temperature condition that we assume. We assume the MHT condition (see \eqref{eq:mht}) for most results, which involves computing the often intractable $\|\cdot\|_4$ operator norm. It would be interesting to replace this with a milder condition, such as the WHT (see \eqref{eq:wht}) condition, or even go beyond the high-temperature regime entirely. 

Another intriguing direction is to prove such CLTs beyond quadratic interaction models. These arise naturally in high-dimensional generalized linear models \cite{McCullagh1989}, Gibbs posteriors \cite{Alquier2016}, and fractional posteriors \cite{Yang2020}. While quadratic interactions seem crucial in our approach, it was recently shown in \cite{mukherjee2024naive} that the posterior under generalized linear models does admit a ``locally" quadratic approximation. We hope to leverage such approximations in combination with our current techniques to go beyond the quadratic interaction setting.
\\

\noindent \textbf{Organization of proofs} \quad %
We prove the main Theorems \ref{thm:LLN}, \ref{thm:clt}, and \ref{thm:CLT for u} in \cref{sec:proof of mainres}. 
In \cref{sec:pfisingrandom}, we prove the theoretical claims for our results on random field Ising models. \cref{sec:pfmainlem} is devoted to the proofs of technical lemmas that were previously used to prove our main results. Finally, we prove some technical matrix lemmas in the Appendix \cref{sec:pfauxlem}.

\section{Proof of main results}\label{sec:proof of mainres}
\begin{proof}[Proof of Lemma \ref{lem:uniqueness of optimizers}]
   Recall the definition of $I_i(z_i) = z_i \phi_i(z_i) - \psi_i(\phi_i(z_i))$, for $1\le i\le n$, from \cref{def:expfam}. Observe that
	\begin{align*}
		I_i'(z_i) = \phi_i(z_i), \quad I_i''(z_i) = \phi_i'(z_i) = \frac{1}{\psi_i''(\phi_i(z_i))} \ge 1.
	\end{align*}
	The inequality follows by noting that
	$\psi_i''(\phi_i(z_i)) = \Var_{\mu_{i,\phi_i(z_i)}} (Z) \le 1$.
	Hence, given any $\bv$ with $\lVert \bv \rVert=1$, we have:
	\begin{align*}
		\bv^{\top}\nabla^2 F(\bz)\bv & = \left( \bv^{\top} \A_n\bv - \sumin \frac{v_i^2}{\psi_i''(\phi_i (z_i))}\right) \\ & \le  \, \lVert \A_n\rVert-1 \le \rho-1< 0, 
	\end{align*}
	where the final inequality follows from the WHT condition \eqref{eq:wht}. Therefore $F(\cdot)$ is strictly concave, and so, there exists a unique maximizer $\bu$ in the convex domain $[-1,1]^n$. This maximizer cannot be at the boundary of $[-1, 1]^n$, since $$\frac{\partial}{\partial z_i} F(\bz) =  \sumjn \A_n(i,j) z_j \red{+ c_i} - \phi_i(z_i)$$
	for $1\le i\le n$, and $\phi_i(z_i) \to \pm \infty$ as $z_i \to \pm 1$.
	The fixed point equations in \eqref{eq:fpeq} now follow by setting $\nabla F(\bz) = 0$. 
\end{proof}

\begin{defi}\label{def:conex}
			For a bounded measurable function $g(\cdot)$ on $[-1,1]$, let $\mathbf{b}^{(g)}:=(b^{(g)}_1,\cdots,b^{(g)}_n)^\top \in [-1,1]^n$ denote the vector of conditional expectations of $g(\cdot)$ under model \eqref{eq:model}, defined by
			$$b^{(g)}_i:=\EE[g(\sigma_i) | \sigma_{j},\ j\neq i].$$
   In the special case where $g(x)=x$, we simplify notation from $\mathbf{b}^{(g)}$ to $\mathbf{b}$, i.e., 
   $$b_i := \EE[\sigma_i | \sigma_j,\ j\neq i] = \psi_i'(m_i+c_i),$$
   where $\{m_i\}_{1\le i\le n}$ are the local fields defined in \eqref{eq:locfield}.
\end{defi} 

 \noindent Our first lemma provides concentration results for linear combinations of $g(\sigma_i)$s which are conditionally centered. 
	
\begin{lemma}\label{lem:sum L2 bound}
		Suppose $\A_n$ satisfies $\lVert \A_n\rVert\le h < \infty$. Let $\bgamma := (\gamma_1,\ldots ,\gamma_n)^\top \in \R^n$ be a deterministic $n$-dimensional vector.  Then, for any bounded measurable function $g(\cdot)$ on $[-1,1]$, the following inequalities hold:
		\begin{enumerate}[(a)]
			\item For any $t > 0$, $$\PP \left(\bigg|\sumin \gamma_i \big(g(\sigma_i) - b^{(g)}_i\big)\bigg| > t \right) \leq 2 \exp\left(-\frac{a t^2}{\|\red{\bgamma}\|^2}\right)$$
			for some constant $a>0$ depending only on $h$ \red{and $\|g\|_\infty$}. Consequently, for $r>0$, we have $$\EE\bigg|\sumin \gamma_i \big(g(\sigma_i) - b^{(g)}_i\big)\bigg|^r \lesssim \left(\frac{\|\bgamma\|}{\sqrt{2a e}}\right)^r r^{\frac{r+1}{2}}.$$
		\item Moreover, $$\EE\bigg|\sumin \gamma_i \big(g(\sigma_i) - b^{(g)}_i\big) b^{(g)}_i\bigg|^2 \lesssim \sumin \gamma_i^2.$$
  \noindent The hidden constant here depends on $h$ and \red{and $\|g\|_\infty$} but not on $\bgamma$.
	\end{enumerate}
\end{lemma}

While similar results have been established in the literature (see for e.g.~\cite[Lemma 2.1]{ghosal2020joint} and \cite[Lemma 2.10]{bhattacharya2023gibbs}), the key distinguishing features of \cref{lem:sum L2 bound} are that we impose no restrictions on the external field $\bc$ and the weight vector $\bgamma$. 

For the next lemma, recall the definition of $\mathbf{m}$ and $\mathbf{s}$ from \eqref{eq:locfield} and \cref{lem:uniqueness of optimizers} respectively. We shall provide high probability bounds on the vector $\bmm-\mathbf{s}$ in different norms depending on the nature of the high-temperature assumption as in \cref{assn:ht}.

\begin{lemma}\label{lem:m-n contraction}
  \vspace{0.1in}   
  \noindent (a) Suppose the WHT condition \eqref{eq:wht} holds. Then, we have $$\EE \sumin (m_i - s_i)^2  \lesssim n \an.$$ 
  \vspace{0.1in}
  \noindent (b) Suppose the MHT condition \eqref{eq:mht} holds. Then, we have $$\EE \sumin (m_i - s_i)^4 \lesssim n \alpha_n^2.$$
   \vspace{0.1in} 
  \noindent (c) Suppose the SHT condition \eqref{eq:sht} holds. %
  Then, with the same constant $a>0$ as in \cref{lem:sum L2 bound}, for any $r\in [2,\infty)$, we have 
	$$\EE \sumin |m_i - s_i|^{r} \lesssim n\left(\frac{\an}{2ae(1-\rho)^2}\right)^{\frac{r}{2}} r^{\frac{r+1}{2}}.$$
\end{lemma}

Now, we are ready to prove \cref{thm:LLN}.

\begin{proof}[Proof of \cref{thm:LLN}]
  We first provide a general tail bound on $T_n^*(\bs)=\sumin  q_i(\sigma_i-u_i)$ which will then be leveraged to prove \eqref{eq:firstbd} and \eqref{eq:secondbd}. Let us assume that $\lVert \A_n\rVert\le \rho<1$. For non-negative integers $\ell \ge 0$ and $1\le j\le n$, recursively define $\bq^{(\ell)}$ by $q_j^{(0)} = q_j$ and
\begin{align}\label{eq:recdef}	
 q_j^{(\ell)} :=  \sumin q_i^{(\ell-1)} \A_n(i,j) \psi_i''( s_i + c_i), 
\end{align}
	where $c_i$s are defined as in \eqref{eq:model}. In other words, setting
 \begin{align*}%
     \mathbf{C}_n := \A_n \text{Diag}(\psi_1''( s_1 + c_1),\ldots , \psi_n''(s_n+c_n)), \quad \mbox{we have }\, \bq^{(\ell)} = \mathbf{C}_n \bq^{(\ell-1)}, \, \ell\ge 1.
 \end{align*}
By a Taylor expansion, for any $\ell\ge 0$ we have
	\begin{equation}\label{eq:b_i - u_i decomposition}
		\begin{aligned}
			\sumin q_i^{(\ell)} (b_i - u_i) &=\sumin q_i^{(\ell)}\Big(\psi_i'(m_i+c_i)-\psi_i'(s_i+c_i)\Big)\\
    &=\sumin q_i^{(\ell)} \left(  (m_i - s_i) \psi_i''( s_i + c_i) + \frac{1}{2} (m_i - s_i)^2 \psi_i'''(\xi_i + c_i) \right) \\
			&= \sumjn q_j^{(\ell+1)} (\sigma_j - u_j) + \frac{1}{2} \sumin q_i^{(\ell)}(m_i - s_i)^2 \psi_i'''(\xi_i + c_i), 
		\end{aligned}
	\end{equation}
 for some $\xi_i$ between $m_i$ and $s_i$.
	For $\ell\ge 0$, setting
 \begin{align*}%
     T_{(\ell)} := \sumin q_i^{(\ell)} (\sigma_i - u_i),
 \end{align*}
	 use \cref{lem:normorder} and the fact that $|\psi_i''(\cdot)|\le 1$ to conclude that 
     \begin{equation}\label{eq:negligible}
     |T_{(\ell)}|\le \red{2} \sqrt{n}\|\mathbf{C}_n\|^\ell \|{\mathbf q}\| \le 2 \sqrt{n}\|\A_n\|^\ell \le 2 \sqrt{n}\rho^\ell  \stackrel{\ell\to\infty}{\to} 0.
     \end{equation}
 Thus we can write
 $$T_n^*(\bs)=\sumin q_i(\sigma_i-u_i) = T_{(0)} = \sum_{\ell=0}^{\infty} (T_{(\ell)}-T_{(\ell+1)}) = \sum_{\ell=0}^{\infty} \sumin (q_i^{(\ell)}-q_i^{(\ell+1)})(\sigma_i-u_i).$$
 By using the triangle inequality and \eqref{eq:b_i - u_i decomposition}, we then have:
 \begin{align*}
     \bigg|\sumin q_i(\sigma_i-u_i)\bigg| & \le \sum_{\ell=0}^{\infty} \bigg|\sumin q_i^{(\ell)}(\sigma_i-u_i)-\sumin q_i^{(\ell+1)}(\sigma_i-u_i)\bigg| \nonumber \\
     & \le \sum_{\ell=0}^{\infty} \bigg|\red{\sumin  q_i^{(\ell)}(\sigma_i-b_i)+ \frac{1}{2}\sumin q_i^{(\ell+1)}(m_i-s_i)^2 \psi_i'''(\xi_i+c_i)} \bigg| \nonumber \\
     &\le \sum_{\ell=0}^{\infty} \bigg|\sumin q_i^{(\ell)}(\sigma_i-b_i)\bigg|+ \frac{1}{2} \sum_{\ell=0}^{\infty}\bigg|\sumin  q_i^{(\ell)}(m_i-s_i)^2\psi_i'''( \xi_i+s_i)\bigg|.
 \end{align*}
Now, for any $\ell\ge 0$, Cauchy-Schwartz inequality along with the bound $|\psi_i'''(\cdot)|\le 8$ gives 
 
 $$\bigg|\sumin q_i^{(\ell)}(m_i-s_i)^2\psi_i'''(\xi_i+s_i)\bigg|\le 8\lVert\bq^{(\ell)} \rVert \lVert \mathbf{m} - \mathbf{s} \rVert_4^2 \le 8\rho^\ell \lVert \mathbf{m} - \mathbf{s} \rVert_4^2.$$
 Combining the above observations, we get: 
 \begin{align}\label{eq:gentailbd}
     &\;\;\;\;\;\PP\left(\bigg|\sumin q_i(\sigma_i-u_i)\bigg|\ge 2t\right)\nonumber \\ &\le \sum_{\ell=0}^{\infty} \PP\left(\bigg|\sumin q_i^{(\ell)}(\sigma_i-b_i)\bigg|\ge t a_\ell\right)+\sum_{\ell=0}^{\infty} \PP\left(4\lVert \mathbf{m} - \mathbf{s} \rVert_4^2 \ge c t (\ell+1)\right).
 \end{align}
In the above display, we set $a_\ell:=c(\ell+1)\rho^\ell$ for $\ell\ge 0$, where the constant $c=c(\rho)$ is chosen such that $\sum_{\ell=0}^\infty a_\ell=1$.
 We will now bound the two terms above. %
 For the first term, we invoke \cref{lem:sum L2 bound}(a), to get:
 \begin{align}\label{eq:triedandtested}
   \notag  \sum_{\ell=0}^{\infty} \PP\left(\bigg|\sumin q_i^{(\ell)}(\sigma_i-b_i)\bigg|\ge t a_l\right)&\le  \red{2\sum_{\ell=0}^{\infty} \exp\left(-a \frac{t^2 a_\ell^2}{\rho^{2\ell}}\right)} \\
   &=\red{2\sum_{\ell=0}^{\infty} \exp\left(-ac^2 t^2 (\ell+1)^2\right)}\nonumber \\
   &\le \red{ 2\sum_{\ell=0}^{\infty} \exp\left(-ac^2 t^2 (\ell+1) \right)} \nonumber \\
   &\le \red{2\frac{\exp(-ac^2 t^2)}{1-\exp(-ac^2)},}
 \end{align}
where the last inequality holds for $t\ge 1$, and $a>0$ is the constant from \cref{lem:sum L2 bound}.
We now split the proof in two parts. 
 
 \vspace{0.1in}

 \noindent\emph{Proof of \eqref{eq:firstbd}.} By \eqref{eq:gentailbd} and \eqref{eq:triedandtested}, we have 
 \begin{align*}
 \EE \big(T_n^*(\bs)\big)^2 & = 8 \int_0^{\infty} t\, \PP\left(\bigg|\sumin q_i(\sigma_i-u_i)\bigg|\ge 2t\right)\,dt \\ 
 &\le \red{8 \int_0^1 tdt + 16 \int_1^\infty t \frac{\exp(-ac^2 t^2)}{1-\exp(-ac^2)} dt + 8 \sum_{l=0}^\infty \int_0^\infty  t \PP\Big(\frac{4\|\mathbf{m}-\mathbf{s}\|_4^2}{c(\ell+1)} \ge t \Big) dt} \\
 & = 4+\frac{\red{16}}{1-\exp(-ac^2)}\int_{1}^{\infty} t \exp(-ac^2 t^2)\,dt + \frac{\red{64}}{c^2}\EE \|\mathbf{m}-\mathbf{s}\|_4^4\sum_{\ell=0}^{\infty} (\ell+1)^{-2}. 
 \end{align*}
 \red{Here, the first and final equality follow from the identity $\EE X_n^2 = 2 \int_0^\infty t \PP(|X_n| \ge t) dt$.}
 The conclusion now follows by invoking \cref{lem:m-n contraction}(b).

 \vspace{0.1in}
 
 \noindent\emph{Proof of \eqref{eq:secondbd}.} Without loss of generality, we can assume $t\ge \max(1,\sqrt{n}\an)$. Otherwise we can just adjust the constants $h_1,h_2$ to account for this. Let $a>0$ denote the constant from \cref{lem:sum L2 bound}(a) and fix $$\vartheta_n:=\frac{a(1-\rho)^2}{16\sqrt{n}\an}.$$ Next, we note the following sequence of equalities and inequalities with line by line explanations to follow.
 \begin{align*}
     \sum_{\ell=0}^{\infty} \PP &\left(8\lVert \mathbf{m}-\mathbf{s}\rVert_{4}^2 \ge ct (\ell+1)\right) \le \sum_{\ell=0}^{\infty}\exp(-\vartheta_n ct (\ell+1)) \EE \exp\left(8\vartheta_n \|\mathbf{m}-\mathbf{s}\|_4^2\right)\\ &=\sum_{\ell=0}^{\infty} \exp(-\vartheta_n ct(\ell+1))\sum_{r=0}^{\infty} \frac{(8\vartheta_n)^r\EE \big[\sumin (m_i-s_i)^4\big]^{\frac{r}{2}}}{r!}\\ 
     & \le \sum_{\ell=0}^{\infty} \exp(-\vartheta_n ct (\ell+1)) \left(1+\red{8\vartheta_n \sqrt{\EE \sumin (m_i-s_i)^4}} + \sum_{r=2}^{\infty} n^{\frac{r}{2}-1}\frac{(8\vartheta_n)^r}{r!}\EE \sumin (m_i-s_i)^{2r}\right)\\ 
     &\lesssim \sum_{\ell=0}^{\infty} \exp(-\vartheta_n ct (\ell+1)) \left(1+\sum_{r=1}^{\infty} n^{\frac{r}{2}}\frac{(8\vartheta_n)^r}{r!}\left(\frac{\an}{2ae(1-\rho)^2}\right)^r (2r)^{r+\frac12}\right)\\ &\le \frac{1}{\sqrt{\pi}}\sum_{\ell=0}^{\infty} \exp(-\vartheta_n ct(\ell+1))\sum_{r=0}^{\infty} \left(\frac{8\sqrt{n}\an \vartheta_n}{a(1-\rho)^2}\right)^r\\
     &=\frac{\exp(-\vartheta_n ct)}{\sqrt{\pi}(1-\exp\left(-\vartheta_n ct\right))}\sum_{r=0}^\infty 2^{-r}\le  \frac{2}{\sqrt{\pi}}\frac{\exp(-\vartheta_n ct)}{1-\exp\left(-\frac{ac(1-\rho)^2}{16}\right)}.
 \end{align*}
 The first display above follows from Markov's inequality. The second display uses the power series expansion of the exponential function. The third display is a consequence of H\"{o}lder's inequality (for the summands indexed by $r \ge 2$). For the fourth display, we use \cref{lem:m-n contraction}(c) \red{to control both expectations}. The fifth display is a product of the standard lower bound $r!\ge \sqrt{2\pi} r^{r+1/2} e^{-r}$. The final inequality is immediate from the choice of $\vartheta_n$, and the assumption $t\ge \sqrt{n}\an$.
 The above display, combined with \eqref{eq:gentailbd} and \eqref{eq:triedandtested}, concludes the proof.
\end{proof}

Now, we prove \cref{thm:clt}. We first state a lemma that upper bounds the norm of $\mathbf{s}$ by the norm of $\boldsymbol{\nu}$, which will be used to simplify some error terms.
\begin{lemma}\label{lem:optimizer u l2 concentration}
For all $i= 1,\ldots,n$, let $\nu_i = \sumjn\A_n(i,j) \psi_j'(c_j)$ be as in~\cref{thm:clt}, and $\boldsymbol{\nu} = (\nu_1, \ldots, \nu_n)^\top$. Recall the definition of ${\bm s}$ from~\cref{lem:uniqueness of optimizers}.
\begin{enumerate}[(a)]
    \item Suppose the WHT condition \eqref{eq:wht} holds. Then $\lVert \mathbf{s} \rVert \lesssim \lVert \boldsymbol{\nu} \rVert$.
    \item Suppose the MHT condition \eqref{eq:mht} holds. Then $\lVert \mathbf{s} \rVert_4 \lesssim \lVert \boldsymbol{\nu} \rVert_4$.
    \item Suppose the SHT condition \eqref{eq:sht} holds. Then $\lVert \mathbf{s} \rVert_\infty \lesssim \lVert \boldsymbol{\nu} \rVert_\infty$.
\end{enumerate}
\end{lemma}

\begin{proof}[Proof of \cref{thm:clt}]
Without loss of generality, we can assume that the  right hand side of \eqref{eq:ksbd} is bounded by $1$, which means, in particular that $\sqrt{n}\an\le 1$. Therefore, by \cref{thm:LLN}, there exists $C>0$ such that 
\begin{equation}\label{eq:prelbd}
    \EE\bigg|\sumin q_i(\sigma_i-u_i)\bigg|^2\le C.
\end{equation}

We will now split the proof into multiple steps. 

\vspace{0.1in}

 \emph{Step 1: Stein's method.}
	Recall that $\boldsymbol{\epsilon} =\A_n \bq - \lambda_n \bq$ and $T_n^*(\bs)=\bq^{\top}(\bs-\bu)$. Sample $\bs'$ as follows: let $I$ be a randomly sampled index from $\{1,2,\ldots ,n\}$. Given $I=i$, replace $\sigma_i$ with an independently generated $\sigma'_i$ drawn from the conditional distribution of $\sigma_i$ given $\sigma_j,\ j\neq i$. Then $$T_n^*(\bs')=\bq^{\top}(\bs'-\bu), \text{ and so }T_n^*(\bs)-T_n^*(\bs')=q_I(\sigma_I-\sigma_I')=:\Delta_n.$$
	By a Taylor's series expansion of $b_i = \psi_i'( m_i + c_i)$ around $m_i \approx s_i$ (recall that $s_i$ was defined in \cref{lem:uniqueness of optimizers}), we get that for some $\{\xi_i\}_{1\le i\le n}$:
	\begin{align*}
		\EE &(T_n^*(\bs) - T_n^*(\bs') \mid \bs) = \frac{1}{n} \sumin q_i(\sigma_i - b_i) \\
		&= \frac{1}{n} \sumin q_i \left(\sigma_i - u_i - (m_i - s_i)\psi_i''( s_i + c_i) - \frac{1}{2} (m_i-s_i)^2 \psi_i'''( \xi_i + c_i) \right) \\
		&= \frac{T_n^*(\bs)}{n} - \frac{1}{n} \sumin q_i (m_i - s_i) \psi_i''( s_i + c_i) - \frac{1}{2n} \sumin q_i (m_i-s_i)^2 \psi_i'''( \xi_i + c_i) \\
		&= \underbrace{\frac{(1 -  \lambda_n \upsilon_n) T_n^*(\bs)}{n}}_{g(T_n^*(\bs))} - \underbrace{\frac{1}{n} \sumin q_i (m_i - s_i) (\psi_i''( s_i + c_i) - \upsilon_n)}_{H_1} \\
		&\quad - \underbrace{\frac{1}{2n} \sumin q_i (m_i-s_i)^2 \psi_i'''( \xi_i + c_i)}_{H_2} - \underbrace{\frac{ \upsilon_n}{n} \sumin \epsilon_i(\sigma_i - u_i)}_{H_3} .
	\end{align*}
	For the final equality, we are simplifying the second term in the penultimate line using the fact that $\sumin\A_n(i,j) q_i = \lambda_n q_j + \epsilon_j$ so that
	\begin{align*}
		\frac{1}{n} \sumin q_i (m_i - s_i) \psi_i''(s_i + c_i) = \frac{1}{n} \sumin q_i (m_i - s_i)(\psi_i''(s_i + c_i) - \upsilon_n) + \frac{\upsilon_n}{n} \sumjn (\lambda_n q_j + \epsilon_j) (\sigma_j - u_j) .
	\end{align*}
   Note that $|\Delta_n|\le 2\lVert \bq\rVert_{\infty}$, and recall the definition of $W_n$ from \cref{thm:clt}. 

Now we apply \cite{Shao2019}, which provides a KS distance bound between a random variable $S_n^*$ and a standard normal. {Recall from the statement of \cref{thm:clt} that the constants $\rho, \kappa > 0$ satisfy $|\lambda_n|\le \rho<1, \upsilon_n \ge \kappa$. For notational convenience, set $S_n^*(\bolds) := T_n^*(\bolds) \sqrt{\frac{1-\lambda_n \upsilon_n}{\upsilon_n}}$ as a re-scaled version of $T_n^*$  %
and note that
$$\EE(S_n^*(\bolds) - S_n^*(\bolds') \mid S_n^*(\bolds)) = g(S_n^*(\bolds)) - \sqrt{\frac{1-\lambda_n \upsilon_n}{\upsilon_n}} \EE(H_1+H_2+H_3 \mid S_n^*(\bolds)).$$
Set $\delta_n := 2 \lVert \bq\rVert_{\infty} \sqrt{\frac{1-\lambda_n \upsilon_n}{\upsilon_n}}$ so that $|S_n^*(\bolds) - S_n^*(\bolds')| \le \delta_n.$ 
Noting that $\EE|S_n^*(\bolds)| = \EE|T_n^*(\bolds)|\sqrt{\frac{1-\lambda_n \upsilon_n}{\upsilon_n}} \lesssim 1$ by using the preliminary bound \eqref{eq:prelbd}, $\upsilon_n \ge \kappa > 0$, and $ 1-\lambda_n \upsilon_n \le 1+\rho$, we can apply \cite[Corollary 2.1]{Shao2019} (in the second line below), which gives
	\begin{align}
		d_{KS}(T_n^*(\bs) , W_n) &= d_{KS}(S_n^*(\bolds), \mcn(0,1) ) \notag \\
        &\le \EE \left| 1 - \frac{n}{2(1-\lambda_n \upsilon_n)} \EE((S_n^*(\bolds) - S_n^*(\bolds'))^2 \mid S_n^*(\bs)) \right| + \frac{n \EE \left[ \sum_{a=1}^3 |H_a| \right]}{\sqrt{\upsilon_n(1-\lambda_n\upsilon_n)}} + O(\delta_n) \notag \\
        &\lesssim \EE \left| 1 - \frac{n}{2\upsilon_n} \EE(\Delta_n^2 \mid T_n^*(\bs)) \right| + {n \EE \left[ \sum_{a=1}^3 |H_a| \right]}+ \lVert \bq\rVert_{\infty}.\label{eq:Stein}
	\end{align}
    For the final line, we again used $\upsilon_n \ge \kappa >0$ and $1-\lambda_n \upsilon_n \ge 1-\rho>0$ to simplify the denominator.}

\vspace{0.1in}

\emph{Step 2: Bounding the first term}. We note the following equality:
\begin{align}\label{eq:basedis}
\EE[(\sigma'_i)^2|\bs]=\EE[\sigma_i^2|\sigma_j,\ j\neq i]=\psi_i''( m_i+c_i)+(\psi_i'( m_i+c_i))^2.
\end{align}
As $\Delta_n=q_I(\sigma_I-\sigma'_I)$, we have:
\begin{align*}	
 \frac{\red{n}}{2}\EE(\Delta_n^2 \mid \bs) & = \frac{1}{2} \sumin q_i^2 \EE\left(\sigma_i^2 + (\sigma_i')^2-2\sigma_i\sigma'_i \mid \bs \right) \\ &=\frac{1}{2}\sumin q_i^2 \bigg(\sigma_i^2+\psi_i''( m_i+c_i)+(\psi_i'( m_i+c_i))^2-2\sigma_i\psi_i'( m_i+c_i)\bigg)\\ &=\frac{1}{2}\sumin q_i^2\left(\sigma_i^2-\EE[\sigma_i^2|\sigma_j,\ j\neq i]\right)+\sumin q_i^2\psi_i''( m_i+c_i)\\ &\quad -\sumin q_i^2\left(\sigma_i-\psi_i'( m_i+c_i)\right)\psi_i'( m_i+c_i).
 \end{align*}
 In the last two displays, we have used \eqref{eq:basedis}. Recall the definition of $b_i$'s from \cref{def:conex}. By the tower property, the triangle inequality, and the above display, we get:
	
	\begin{align*}
	&\;\;\;\;	\EE \left| 1 - \frac{n}{2\upsilon_n} \EE(\Delta_n^2 \mid T_n^*(\bs)) \right| \\ &\lesssim \frac{1}{\upsilon_n} \bigg(\EE\bigg|\sumin q_i^2\left(\sigma_i^2-\EE[\sigma_i^2|\sigma_j,\ j\neq i]\right)\bigg|+\EE\bigg|\sumin q_i^2(\psi_i''( m_i+c_i)-\upsilon_n)\bigg|\\ &\qquad +\EE\bigg|\sumin q_i^2 \left(\sigma_i-b_i\right)b_i \bigg| \bigg).
	\end{align*}
    Note that first and the last term in the above bound can be controlled directly by invoking \cref{lem:sum L2 bound}, parts (a) and (b) respectively. This will yield: 
 \begin{align*}
     \EE\bigg|\sumin q_i^2\left(\sigma_i^2-\EE[\sigma_i^2|\sigma_j,\ j\neq i]\right)\bigg|+\EE\bigg|\sumin  q_i^2\EE[\left(\sigma_i-b_i\right)b_i]\bigg|\lesssim \sqrt{\sumin q_i^4}\le \lVert \bq\rVert_{\infty}.
 \end{align*}

 For the second term, we note that by the definition of $\upsilon_n$, we have
 \begin{align*}
     &\;\;\;\;\EE\bigg|\sumin q_i^2(\psi_i''( m_i+c_i)-\upsilon_n)\bigg|\\ &\le \EE\bigg|\sumin  q_i^2(\psi_i''( m_i+c_i)-\psi_i''( s_i+c_i))\bigg|+\bigg|\sumin q_i^2(\psi_i''( s_i+c_i)-\psi_i''(c_i))\bigg|
 \end{align*}
To bound the first term above, we use H\"{o}lder's inequality with exponents $4/3$ and $4$, with \cref{lem:m-n contraction}(b), to get
    \begin{align*}
    \EE\bigg|\sumin q_i^2(\psi_i''( m_i+c_i)-\psi_i''( s_i+c_i))\bigg| &\lesssim \EE \sumin q_i^2 |m_i-s_i| \\
    &\le \left(\sumin q_i^{8/3} \right)^{3/4} \left(\EE\sumin (m_i-s_i)^4 \right)^{1/4} \\
         &\lesssim \lVert \bq\rVert_{\infty}^{1/2} \left(n \alpha_n^2\right)^{1/4} \lesssim  \lVert \bq\rVert_{\infty} + \sqrt{n\alpha_n^2}.
    \end{align*}
For the second term, we can repeat the same argument as above and apply \cref{lem:optimizer u l2 concentration} to get:
\begin{align*}
    \bigg|\sumin q_i^2(\psi_i''( s_i+c_i)-\psi_i''(c_i))\bigg|& \lesssim \left(\sumin q_i^{8/3} \right)^{3/4} \left(\sumin s_i^4 \right)^{1/4}\\ & \lesssim \lVert \bq\rVert_{\infty}^{1/2}\left(\sumin  \nu_i^4\right)^{1/4}\lesssim \lVert \bq\rVert_{\infty}+\sqrt{R_{3n}}.
\end{align*}
By combining the relevant bounds from above and recalling that $\upsilon_n \ge \kappa$, we get:
$$\EE \left| 1 - \frac{n}{2\upsilon_n} \EE(\Delta_n^2 \mid T_n^*(\bs)) \right|\lesssim \left(\lVert \bq\rVert_{\infty}+\sqrt{n\alpha_n^2}+\sqrt{R_{3n}}\right).$$
 
	\emph{Step 3: Controlling $H_a$'s.}
	We will bound
	$n^2 \EE|H_a|^2$s for $a=1,2,3$. 
	Let us begin with $H_1$. Define 
 $$\gamma_i := \sumjn\A_n(i,j) q_j (\psi_j''( s_j + c_j) - \upsilon_n)$$ and note that
	$$n^2 \EE |H_1|^2 = \EE\bigg|\sumin \gamma_i (\sigma_i - u_i)\bigg|^2\lesssim \sumin \gamma_i^2,$$
	where the last inequality follows from \cref{thm:LLN}. \red{Note that we used $\max(1,n\alpha_n^2) = 1$ to simplify the RHS, which follows from the assumption $\sqrt{n} \an \le 1$ (see one line before \eqref{eq:prelbd}).} In order to bound $\sumin \gamma_i^2$, we note that for some $\{\xi_j\}_{1\le i\le n}$,
	\begin{align*}
		\gamma_i^2 &= \bigg|\sumjn\A_n(i,j) q_j (\psi_j''(c_j) - \upsilon_n) +  \sumjn\A_n(i,j) q_j s_j \psi_j'''( \xi_j + c_j)\bigg|^2 \\
		&\le 2\bigg|\sumjn\A_n(i,j) q_j (\psi_j''(c_j) - \upsilon_n)\bigg|^2 + \red{128} \left(\sumjn\A_n(i,j)^2 q_j^2 \right)
		\left(\sumjn s_j^2\right),
	\end{align*}
	where we use the bound $|\psi_j'''(\cdot)| \le 8$. Therefore, by applying \cref{lem:optimizer u l2 concentration}, we get: 
	\begin{align*}
		\sumin \gamma_i^2 &\lesssim \sumin \bigg|{\sumjn\A_n(i,j) q_j (\psi_j''(c_j) - \upsilon_n)}\bigg|^2 + \left(\sumin \sumjn\A_n(i,j)^2 q_j^2\right) \left(\sumjn  \red{\nu_j^2}\right) \\
		&= R_{1n} + \sumin \sumjn\A_n(i,j)^2 q_j^2 R_{2n} 
		\le R_{1n} + \an R_{2n}.
	\end{align*}
 To conclude, we have proved

 $$n^2 \EE|H_1|^2 \lesssim R_{1n}+\an R_{2n}.$$
	We now move on to $H_2$. 
	From \cref{lem:m-n contraction}(b), we have $\EE\sumin (m_i-s_i)^4\lesssim n\alpha_n^2$. Combining this observation with the Cauchy Schwartz inequality, we have: 
 $$n^2\EE|H_2|^2 \lesssim \EE \left(\sumin |q_i| |m_i - s_i|^2\right)^2 \le \lVert \bq\rVert^2 {\EE \sumin (m_i - s_i)^4} \lesssim n \alpha_n^2.$$

    Finally, we bound $H_3$. Recall that $\boldsymbol{\epsilon} =\A_n \bq - \lambda_n \bq$. We again use \cref{thm:LLN} \red{and $\max(1,n\alpha_n^2) = 1$} to see
	$$n^2 \EE |H_3|^2 \lesssim \lVert\boldsymbol{\epsilon} \rVert^2.$$

\vspace{0.1in} 

By combining the above bounds, we get:

$$n^2 \sum_{a=1}^3 \EE |H_a|^2 \lesssim R_{1n} + \an R_{2n} + n \alpha_n^2 + \lVert\boldsymbol{\epsilon} \rVert^2.$$
	\noindent Combining the bounds from the previous two steps with \eqref{eq:Stein} completes the proof.
	\end{proof}

\noindent Next, we prove \cref{thm:CLT for u}. For proving part (b), we use the following simple lemma, whose proof we defer to the appendix.
    \begin{lemma}\label{lem:ksbound}
Fixing $K>1$,  for any random variable $X$ and constants $a>0,b\in \R$ and $\tau\in [K^{-1},K]$ we have
\begin{align}\label{eq:ksbound}d_{KS}(aX+b,\mcn(0,\tau))\lesssim |a^2-1|+|b|+d_{KS}(X,\mcn(0,\tau))
        \end{align}
    where the hidden constant only depends on $K$.
    \end{lemma}

\begin{proof}[Proof of \cref{thm:CLT for u}]
\begin{enumerate}
    \item[(a)]

    Recall that $\bu$ satisfies the fixed point equation $u_i = \psi_i'( s_i+ c_i)$, where $s_i=\sumjn\A_n(i,j)u_j$ as in Lemma \ref{lem:uniqueness of optimizers}. By a Taylor expansion, we can write
    \begin{align*}
        u_i = \psi_i'(c_i) +  s_i \psi_i''(c_i) + e_i, ~\forall i,
    \end{align*}
    where $e_i:= \frac{s_i^2}{2} \psi_i'''( \xi_i + c_i)$, for some reals $\{\xi_i\}_{1\le i\le n}$. Then, with $\Psi'(\bc) = (\psi_1'(c_1), \ldots, \psi_n'(c_n))^{\red{\top}}$, we have
    \begin{align*}
        \bq^\top \bu &= \bq^\top \Psi'(\bc) +  \sumin s_i q_i \psi_i''(c_i) + \bq^\top \mathbf{e} \\
         &= \bq^\top \Psi'(\bc) +  \sumij\A_n(i,j) u_j q_i \psi_i''(c_i)   + \bq^\top \mathbf{e}\\
        &= \bq^\top \Psi'(\bc) +  \sumij\A_n(i,j) u_j q_i (\psi_i''(c_i) - \upsilon_n) +  \upsilon_n \sumjn u_j (\lambda_n q_j + \epsilon_j) + \bq^\top \mathbf{e},
    \end{align*}
    where in the last line above we use the identity $\sumin\A_n(i,j) q_i = \lambda_n q_j + \epsilon_j$. 
    By adjusting the terms, we see that 
    \begin{align}\label{eq:clt for optimizer}
        (1 -  \lambda_n \upsilon_n) \bq^\top \bu - \bq^\top \Psi'(\bc) =  \sumij\A_n(i,j) u_j q_i (\psi_i''(c_i) - \upsilon_n) + \bq^\top \mathbf{e} + \upsilon_n \bu^\top \boldsymbol{\epsilon}.
    \end{align}
    Since $1-\lambda_n\upsilon_n\ge 1-\rho > 0$, it suffices to bound the terms in the RHS of \eqref{eq:clt for optimizer}. Towards this direction, note that the first term can be controlled by writing
    \begin{align*}
        &\bigg|\sumij\A_n(i,j) u_j q_i (\psi_i''(c_i) - \upsilon_n)\bigg| \\
        \le & \bigg|\sumij\A_n(i,j) \psi_j'(c_j) q_i (\psi_i''(c_i) - \upsilon_n)\bigg| +\bigg| \sumjn (u_j - \psi_j'(c_j)) \left[\sumin\A_n(i,j) q_i (\psi_i''(c_i) - \upsilon_n)\right]\bigg| \\
        \le& R_{4n} + \lVert \bu - \Psi'(\bc)\rVert \sqrt{R_{1n}} \lesssim R_{4n} +  \sqrt{R_{1n} R_{2n}}.
    \end{align*}
    Here, the last inequality follows from using \cref{lem:optimizer u l2 concentration}(a) and noting that $$\lVert \bu - \Psi'(\bc)\rVert^2 \,=\lVert  \Psi'({\bf s}+\bc) - \Psi'(\bc)\rVert^2 \,\le\sumin s_i^2 \,\lesssim \sumin \nu_i^2 \, =\, R_{2n}.$$ 
    The second term on the right hand side of \eqref{eq:clt for optimizer} can be bounded using the Cauchy-Schwartz inequality and \cref{lem:optimizer u l2 concentration}(b) as follows: 
    $$\red{|\bq^\top \be|} \,\le \|\be\| \lesssim \sqrt{\sumin s_i^4} \, \lesssim \sqrt{\sumin \nu_i^4} \, = \sqrt{R_{3n}}\,.$$ The bound for the third term follows similarly by combining the Cauchy-Schwartz inequality with \cref{lem:optimizer u l2 concentration}(b), i.e.,
    $$\big|\bu^{\top}\boldsymbol{\epsilon}\big|\le \big|(\bu - \Psi'(\bc))^{\top}\boldsymbol{\epsilon}\big|+\big|\boldsymbol{\epsilon}^{\top} \Psi'(\bc)\big| \le \lVert \mathbf{s}\rVert \lVert \boldsymbol{\epsilon}\rVert + \big|\boldsymbol{\epsilon}^{\top} \Psi'(\bc)\big| \lesssim \sqrt{R_{2n}}\lVert \boldsymbol{\epsilon}\rVert + \big|\boldsymbol{\epsilon}^{\top} \Psi'(\bc)\big|.$$
    The proof of \eqref{eq:CLT for u} then follows by plugging these bounds in \eqref{eq:clt for optimizer}.    

\item[(b)]

Proceeding to verify part (b), {recalling that $|\lambda_n| \le \rho < 1$ and $\upsilon_n \le 1$, we have
$$1-\rho \le 1-\lambda_n \upsilon_n \le 1+\rho.$$
Also recall that we assume $\upsilon_n \ge \kappa$}.
Invoking Lemma \ref{lem:ksbound} with the choices $$X = \bq^\top(\boldsymbol{\sigma}-\bu), \quad\tau=\frac{\red{\upsilon_n}}{1-\lambda_n \upsilon_n}\in [\kappa(1-\rho),1+\rho],\quad  a=1,\quad  b=\bq^{\top}\left(\bu - \frac{\Psi'(\bc)}{1-\lambda_n \upsilon_n}\right)$$ gives
    \begin{align*}
        &\;\;\;\;d_{KS}\bigg(\bq^{\top}\bs-\frac{\bq^{\top}\Psi'(\bc)}{1-\lambda_n\upsilon_n} \, ,\, W_n\bigg) 
        \lesssim d_{KS}\bigg(\bq^{\top}(\bs-\bu) \, ,\, W_n\bigg)+
        \bigg| \bq^{\top}\left(\bu - \frac{\Psi'(\bc)}{1-\lambda_n \upsilon_n}\right)\bigg|, 
    \end{align*}
    where the hidden constants depend on $\kappa, \rho$.  
    Part (b) now follows on plugging-in the bounds in (i) \red{\eqref{eq:ksbd} in \cref{thm:clt}}, and (ii) part (a) of this theorem.
       \end{enumerate}
\end{proof}

\section{Proof of results in \cref{sec:isingrandom}}\label{sec:pfisingrandom}

\begin{proof}[Proof of \cref{cor:i.i.d. clt}]
Recall that $c_1, c_2, \ldots , c_n \overset{i.i.d.}{\sim}\, F$. We need to provide upper bounds to the RHS of  \eqref{eq:ksbd} and \eqref{eq:CLT for u}. %
We first claim that the conclusions of parts (a) and (b) both follow if we can show the following bounds:
\begin{equation}
\begin{aligned}\label{eq:suffice1}
\EE_F (\upsilon_n - \upsilon)^2 &= O(\|\bq\|_\infty^2), \quad \EE_F R_{1n}=O(\alpha_n+\|{\bm \epsilon}\|^2), \quad \EE_F R_{2n} =O(n\alpha_n), \\  \quad \EE_F R_{3n}&=O(n\alpha_n^2), \quad \EE_F R_{4n}=O(\sqrt{\an} + \|\bq\|_\infty), \quad \EE_F|{\bm \epsilon}^\top \Psi'({\bf c}) | =O(\|\bm \epsilon\|).
\end{aligned}
\end{equation}
We defer the proof of \eqref{eq:suffice1} to the end of the proof. \red{In \eqref{eq:suffice1} and throughout this proof, the implicit constants in $O(\cdot)$ and $\lesssim$ may depend on $\rho$ from \eqref{eq:mht} and $\upsilon$.}

\begin{enumerate}
    \item[(a)]

    To begin, setting $$T_n^*(\bs)=\sum_{i=1}^nq_i(\sigma_i-u_i)\quad \text{ and } \quad W_n\sim \mcn\Big(0,\frac{\upsilon_n}{1-\lambda \upsilon_n}\Big)$$ as before, 
 triangle inequality gives
\begin{align}\label{eq:triangle}
 d_{KS}(T_n^*(\bs), W_\infty\mid {\bf c})\le & d_{KS}(T_n^*(\bs),W_n\mid {\bf c})+d_{KS}(W_n,W_\infty\mid {\bf c}).
\end{align}
Using \cref{thm:clt}, on the set $\upsilon_n\ge \kappa := \upsilon/2$, the first term in the RHS of \eqref{eq:triangle} can be bounded as follows:
    \begin{align}\label{eq:qu1}
    d_{KS}(T_n^*(\bs),W_n\mid {\bf c})\lesssim \sqrt{R_{1n}}+\sqrt{\alpha_n R_{2n}}+\sqrt{R_{3n}}+\sqrt{n}\alpha_n+\|{\bf q}\|_\infty+\|{\bm \epsilon}\|=:S^{(1)}_n.
    \end{align}
    The second term in the RHS of \eqref{eq:triangle} can be bounded as follows:
\begin{align}\label{eq:qu2}
     d_{KS}(W_n,W_\infty\mid {\bf c})=d_{KS}\left(\mcn\Big(0, \frac{\upsilon_n}{1-\lambda \upsilon_n}\Big), \mcn\Big(0,\frac{\upsilon}{1-\lambda \upsilon}\Big)\Big|{\bf c}\right)
    \lesssim |\upsilon_n-\upsilon|\lesssim \|{\bf q}\|_\infty,
\end{align}
where we use Lemma \ref{lem:ksbound}, along with the fact that $1-\lambda \upsilon_n\in [ 1-\rho, 1+\rho]$. The desired conclusion follows by combining the above three displays, on using \eqref{eq:suffice1} to note that  $$\EE_F[S^{(1)}_n]\lesssim \textsc{Err}(\A_n, \bq),\quad \text{ and }\quad \PP(\upsilon_n\ge \kappa)\to 1. $$

\vspace{3mm}
\item[(b)]
For the first conclusion (quenched bound), using Lemma \ref{lem:ksbound} with the choices
$$X = \bq^{\top} \left(\bs-\frac{\Psi'(\bc)}{1-\lambda \red{\upsilon_n}}\right), ~ \tau=\frac{\red{\upsilon}}{1-\lambda \upsilon},~  a=1,~  b=\bq^{\top}\left(\red{\frac{\Psi'(\bc)}{1-\lambda \upsilon_n}} - \frac{\Psi'(\bc)}{1-\lambda \upsilon}\right)$$ 
gives
\begin{align}\label{eq:triangle2}
  \notag  &d_{KS}\left(\sum_{i=1}^nq_i\Big(\sigma_i-\frac{\psi'(c_i)}{1-\lambda \upsilon}\Big), W_\infty\mid {\bf c}\right)\\
  \notag  \lesssim& d_{KS}\left(\sum_{i=1}^nq_i\Big(\sigma_i-\frac{\psi'(c_i)}{1-\lambda \upsilon_n}\Big),W_\infty\mid {\bf c}\right)+\left|\sum_{i=1}^nq_i\psi'(c_i)\Big(\frac{1}{1-\lambda \upsilon}-\frac{1}{1-\lambda \upsilon_n}\Big)\right|\\
    \lesssim &S^{(2)}_n+\left|(\upsilon_n-\upsilon)\sum_{i=1}^nq_i\psi'(c_i)\right|,
\end{align}
where in the last inequality we use \cref{thm:CLT for u} part (b), and restrict ourselves to the set $\upsilon_n\ge \kappa = \upsilon/2$,  and 
\begin{align*}
    S^{(2)}_n:=\sqrt{R_{2n}}\left(\sqrt{R_{1n}} + \sqrt{\alpha_n} + \lVert \boldsymbol{\epsilon}\rVert\right) + \sqrt{R_{1n}} +  \sqrt{R_{3n}}  + \sqrt{n} \an + \lVert \bq\rVert_{\infty}+ R_{4n} +\big|\boldsymbol{\epsilon}^{\top}\Psi'(\bc)\big|+\red{\|\boldsymbol{\epsilon}\|}.
\end{align*}
Using \eqref{eq:suffice1} we have $|\upsilon_n-\upsilon|\lesssim \|{\bf q}\|_\infty$, and Lindeberg-Feller CLT gives $\sum_{i=1}^nq_i\psi'(c_i)=O_P(1)$. Finally, again using \eqref{eq:suffice1} gives $$\EE_F[S^{(2)}_n]\lesssim \textsc{Err}(\A_n, \bq)+\sqrt{n\alpha_n}\|{\bm \epsilon}\|.$$
Combining the above, it follows that
$$d_{KS}\left(\sum_{i=1}^nq_i\Big(\sigma_i-\frac{\psi'(c_i)}{1-\lambda \upsilon}\Big), W_\infty  \red{\mid {\bf c}} \right)=O_P(\textsc{Err}(\A_n, \bq)+\sqrt{n\alpha_n}\|{\bm \epsilon}\|),$$
as desired.
\\

 Proceeding to verify the second (i.e. annealed) conclusion of part (b), let $(W_\infty,\widetilde{W}_\infty)$ be mutually independent Gaussians, coupled in the same space as $(\bs, {\bf c})$ via an independent coupling. Then, setting $W^*_n:=\frac{1}{1-\lambda\upsilon}\sum_{i=1}^nq_i\psi'(c_i)$, triangle inequality gives
 \begin{align*}%
\notag d_{KS}\left(\sum_{i=1}^nq_i\sigma_i,W_\infty+\widetilde{W}_\infty\right)\le &
d_{KS}\left(\sum_{i=1}^nq_i\sigma_i,W_\infty+W_n^*\right)+d_{KS}\left(W_\infty+W_n^*,W_\infty+\widetilde{W}_\infty\right)\\
 \le & \EE_F \left[d_{KS}\left(\sum_{i=1}^nq_i\sigma_i,W_\infty+W_n^*\mid {\bf c}\right)\right]+d_{KS}\left(W_n^*,\widetilde{W}_\infty\right),
 \end{align*}
 where in the last inequality we use Jensen's inequality, and the fact that $W_n^*, W_\infty, \widetilde{W}_\infty$ are mutually independent. Since given ${\bf c}$ the random variable $W_n^*$ is a constant, the first term in the RHS above can be bounded as follows:
\begin{align*}d_{KS}\left(\sum_{i=1}^nq_i\sigma_i,W_\infty+W_n^*\mid {\bf c}\right)=&d_{KS}\left(\sum_{i=1}^nq_i\sigma_i-W_n^*,W_\infty\mid {\bf c}\right)\\
=&d_{KS}\left(\sum_{i=1}^nq_i\Big(\sigma_i-\frac{\psi'(c_i)}{1-\lambda \upsilon}\Big),W_\infty\mid {\bf c}\right)\\
\lesssim & S_n^{(2)}+\left|(\upsilon_n-\upsilon)\sum_{i=1}^nq_i\psi'(c_i)\right|,
\end{align*}
where the last inequality holds on the set $\upsilon_n\ge \kappa = \upsilon/2$, and we have used \eqref{eq:triangle2}. Noting that $d_{KS}$ is always bounded by $1$, the above two displays give
\begin{align}\label{eq:triangle3}
    d_{KS}(\sum_{i=1}^nq_i\sigma_i,W_\infty+\widetilde{W}_\infty)\lesssim \EE_F S_n^{(2)}+\EE_F \left|(\upsilon_n-\upsilon)\sum_{i=1}^nq_i\psi'(c_i)\right|+\PP(\upsilon_n<\kappa)+d_{KS}\left(W_n^*,\widetilde{W}_\infty\right).
\end{align}
Using \eqref{eq:suffice1} we can bound the first two terms in the RHS of \eqref{eq:triangle3} as follows:
\begin{align*}
 \EE_F S_n^{(2)}\lesssim & \textsc{Err}(\A_n, \bq)+\sqrt{n\alpha_n}\|{\bm \epsilon}\|,\\
 \EE_F\left|(\upsilon_n-\upsilon)\sum_{i=1}^nq_i\psi'(c_i)\right|\le &\sqrt{\EE_F(\upsilon_n-\upsilon)^2 }\sqrt{\EE_F \Big[\sum_{i=1}^n q_i\red{\psi'}(c_i)\Big]^2}\lesssim \|{\bf q}\|_\infty.
 \end{align*}
 For the third term, Markov's inequality gives
 \begin{align*}
     \PP(\upsilon_n<\kappa)\le \PP(|\upsilon_n-\upsilon|>\upsilon/2)\lesssim \EE_F |\upsilon_n-\upsilon|\lesssim \|{\bf q}\|_\infty,
 \end{align*}
 where the last bound again uses \eqref{eq:suffice1}. Finally, the fourth term in the RHS of \eqref{eq:triangle3} can be bounded using the standard Berry-Esseen CLT for i.i.d. random variables, which gives
$$d_{KS}\Big(\frac{\sumin q_i \psi'(c_i)}{1-\lambda \upsilon}, \widetilde{W}_\infty \Big) \lesssim \sumin |q_i|^3 \le \|\bq\|_\infty.$$
The desired annealed bound follows on combining the last four bounds, along with \eqref{eq:triangle3}.
\end{enumerate}
\vspace{3mm}
\noindent \emph{Proof of \eqref{eq:suffice1}.}
First we bound $\upsilon_n - \upsilon$. We use the independence of $\bc=(c_1,\ldots,c_n)^\top$ and the definition of $\upsilon = \EE_F [\psi''(c_1)]>0$ to get
    \begin{align}\label{eq:latecall}
    \EE_F(\upsilon-\upsilon_n)^2=\mbox{Var}(\psi''(c_1))\sumin q_i^4 \le \sumin q_i^4 \le \|\bq\|_\infty^2.
    \end{align}

Next, we bound $R_{1n}$. Similarly, using the independence of $\bc$, we have
    \begin{align*}
        \EE_F R_{1n} &= \sumin \EE_F \left(\sumjn \A_n(i,j) q_j (\psi''(c_j) - \upsilon_n) \right)^2 \\ 
        &\le 2\EE_F \sumin \left(\sumjn \A_n(i,j) q_j (\psi''(c_j) - \upsilon) \right)^2 + 2\sumin \left(\sumjn \A_n(i,j) q_j \right)^2 \EE_F(\upsilon - \upsilon_n)^2 \\
        &=2\Var(\psi''(c_1)) \sumin \sumjn \A_n(i,j)^2 q_j^2 + 2\lVert \A_n \bq \rVert^2 \EE_F(\upsilon - \upsilon_n)^2 
    \end{align*}
    For the first term above, we note that $\sumin \sumjn \A_n(i,j)^2 q_j^2 \le \an$. For the second term, using the bound $\|\A_n \bq\|^2=\|\boldsymbol{\epsilon}+\lambda \bq\|^2\le 2\|\boldsymbol{\epsilon}\|^2+2\lambda^2$ alongside \eqref{eq:latecall} gives
    \begin{align}
        \lVert \A_n \bq \rVert^2 \EE_F(\upsilon - \upsilon_n)^2 &\le 2\lVert\boldsymbol{\epsilon}\rVert^2\sumin q_i^4 + 2\sumin (\lambda q_i)^2 q_i^2 \nonumber \\ & \le 2\lVert \boldsymbol{\epsilon}\rVert^2 + 4\sumin \epsilon_i^2 q_i^2 + 4\sumin \left(\sumjn \A_n(i,j)q_j\right)^2 q_i^2\le 6\lVert \boldsymbol{\epsilon}\rVert^2+4\alpha_n.\label{eq:upsilon_n bound}
    \end{align}
    The bounds for $R_{2n}$ and $R_{3n}$ also follow similarly by explicit computations leveraging the independence of $\bc$ and $\EE_F[ \psi'(c_1)]=0$, as follows: %

    \begin{align*}
        \EE_F R_{2n} =\sum_{i=1}^n \EE_F\left(\sum_{j=1}^n \A_n(i,j)\psi'(c_j)\right)^2 =  \EE_F [\psi'(c_1)]^2\sum_{i,j=1}^n \A_n^2(i,j) \le n \alpha_n.
    \end{align*}

    \begin{align*}
        \EE_F R_{3n} = \sum_{i=1}^n \EE_F\left(\sum_{j=1}^n \A_n(i,j)\psi'(c_j)\right)^4 \lesssim \sum_{i=1}^n \left(\sum_{j=1}^n \A_n^2(i,j)\right)^2 \le n \alpha_n^2.
    \end{align*}
    
Next we focus on $R_{4n}$ first. By the triangle inequality, we have 
    $$R_{4n} \le \bigg|\sum_{i,j=1}^n \A_n(i,j)q_i(\psi''(c_i)-\upsilon)\psi'(c_j)\bigg|+\bigg|(\upsilon-\upsilon_n)\sum_{i,j=1}^n \A_n(i,j)q_i\psi'(c_j)\bigg|. $$
    We will show that the second moment of the both terms converge to $0$. For the first term,  we have %
    \begin{align*}
        &\;\;\;\;\EE_F\Big[\sumij \A_n(i,j) q_i (\psi''(c_i) - \red{\upsilon}) \psi'(c_j)\Big]^2 
        \\ &= \sum_{i,j,k,l} \A_n(i,j) \A_n(k,l) q_i q_k  \EE_F \left[(\psi''(c_i) - \upsilon)(\psi''(c_k) -\upsilon) \psi'(c_j) \psi'(c_l) \right] \\
         &\lesssim \sum_{i\neq j} \A_n(i,j)^2 \bigg[q_i^2 \EE_F \left[(\psi''(c_i) - \upsilon)^2 \psi'(c_j)^2 \right]+|q_i q_j| \EE_F \red{\Big|} \psi'(c_i) (\psi''(c_i)-\upsilon)\psi'(c_j) (\psi''(c_j)-\upsilon)\red{\Big|}  \bigg]\\
         &\lesssim  \sum_{i\neq j} \A_n(i,j)^2 \bigg[q_i^2 +|q_i q_j|\bigg]
         \le 2\alpha_n.
    \end{align*}
    For the second term, use the MHT condition \eqref{eq:mht} along with Lemma \ref{lem:normorder} to note that $\|\A_n\|\le 1$, and so,
    \begin{align*}
        \EE_F\bigg[(\upsilon-\upsilon_n)\sum_{i,j=1}^n \A_n(i,j)q_i\psi'(c_j)\bigg]^2 &\le {\EE_F(\upsilon-\upsilon_n)^2}{\EE_F\left(\sum_{i,j=1}^n \A_n(i,j)q_i\psi'(c_j)\right)^2}\\ &\le {\sumin q_i^4}\lVert \A_n\bq\rVert^2 \le \lVert \bq\rVert_{\infty}^2\to 0.
    \end{align*}
    The next term to bound is $|\boldsymbol{\epsilon}^{\top}\Psi'(\bc)|$. 
    Once again, the independence of $\bc$ implies that 
    $$\EE_F \big|\boldsymbol{\epsilon}^{\top}\Psi'(\bc)\big|\le \sqrt{\EE_F\left(\sum_{i=1}^n \epsilon_i \psi'(c_i)\right)^2} \le \lVert \boldsymbol{\epsilon}\rVert.$$
    This completes the proof.
\end{proof}

In the proof of the applications that follow, we only need to verify the approximate eigenvector-eigenvalue pair condition $\|{\bf A}_n{\bf q}-\lambda {\bf q}\|=o\Big(\frac{1}{\sqrt{n\alpha_n}}\Big)$, the SHT condition \eqref{eq:sht} (stronger than MHT condition) and $\alpha_n\ll n^{-1/2}$. The desired conclusions then follow from~\cref{cor:i.i.d. clt}.
\begin{proof}[Proof of \cref{lem:ER}]
 Note that $d_i$, the degree of vertex $i$, follows a $Bin(n-1, p_n)$ distribution. Thus
    $$\|{\bf A}_n\|_\infty=\frac{|\theta|}{np_n}\max_{i\in [n]}d_i\stackrel{P}{\to}|\theta|$$ for $p_n\gg n^{-1/2}$ by using standard concentration bounds (such as Hoeffding's inequality), and so the SHT condition holds for $\theta\in (-1,1)$. The above display also gives $$np_n \alpha_n =\red{\theta^2} \max_{i\in [n]}\frac{d_i}{np_n}\stackrel{P}{\to} \red{\theta^2},$$
    and so $\alpha_n=O_P\Big(\frac{1}{{np_n}}\Big) = \red{o_P(n^{-1/2})}$.
    
    \begin{enumerate}[(a).]
    \item 
    With $\bq=n^{-1/2}\mathbf{1}$ and $\lambda=\red{\theta}$, we have $(\A_n \bq - \lambda \bq)_i = \frac{\theta}{\sqrt{n}} \Big(\frac{d_i}{n p_n}-1 \Big)$, we have
    $$\EE\lVert \A_n \bq - \lambda \bq \rVert^2 =  \frac{\theta^2}{n} \sumin \EE\Big(\frac{d_i}{n p_n} - 1\Big)^2 \lesssim \frac{1}{n p_n}.$$
    Therefore, 
    $$\sqrt{n\alpha_n}\lVert \A_n\bq-\lambda \bq\rVert=O_P\left(\sqrt{\frac{\alpha_n}{p_n}}\right)=O_P\left(\frac{1}{\sqrt{n}p_n}\right)\to 0.$$

    \item Now, suppose $\bq$ satisfies \eqref{eq:contrastcon}. For simplicity of notation, set $H_i := \sumjn \G_n(i,j) q_j$, and note that $H_i$ has mean $p_n\sum_{i=1}^n q_i$ and variance $p_n(1-p_n)$. Noting that $(\A_n \bq - \lambda \bq)_i = \frac{\theta H_i}{n p_n}$ for %
    $\lambda=0$, we have
     \begin{align*}
        \EE\lVert \A_n \bq - \lambda \bq \rVert^2 =  \frac{\theta^2}{n^2 p_n^2} \sumin \EE H_i^2 \le \frac{1}{n p_n} + \frac{1}{n}\bigg(\sum_{i=1}^n q_i\bigg)^2.
    \end{align*}
    Therefore,
    $$\sqrt{n\alpha_n}\lVert \A_n\bq-\lambda \bq\rVert = O_P\left(\sqrt{\frac{\alpha_n}{p_n}}+\sqrt{\alpha_n}\bigg|\sum_{i=1}^n q_i\bigg|\right)=O_P\left(\frac{1}{\sqrt{n}p_n}+\frac{1}{\sqrt{np_n}}\bigg|\sum_{i=1}^n q_i\bigg|\right)\to 0.$$
\end{enumerate}
This completes the proof.
\end{proof}

\begin{proof}[Proof of \cref{cor:ER multidimensional}]
    For simplicity, we only show the convergence of two dimensional marginals for part (a) of \cref{cor:ER multidimensional}. The argument directly generalizes to part (b), and we omit the details.
    
    Write $\upsilon := \EE \psi''(c_1)$. By the Cramér–Wold device, it suffices to show that for any $a, b \in \R$ and $Q_1,Q_2\in \mathcal{Q}$,
    $$a\sqrt{n}\langle Q_1,g_{\bs}-g_{
    {\bf u}}\rangle + b\sqrt{n} \langle Q_2,g_{\bs}-g_{
    {\bf u}}\rangle \mid \bc \xd \mcn \Big(0, \upsilon \underbrace{(a^2 + 2ab \int_0^1 Q_1(x) Q_2(x) + b^2)}_{:= F_{a,b}(Q_1, Q_2)}\Big).$$
    Setting $Q(x) := \frac{a Q_1(x) + b Q_2(x) }{\sqrt{F_{a,b}(Q_1, Q_2)}}$ we then have $Q\in \mathcal{Q}$, and
   $$ \sqrt{n}\langle Q,g_{\bs}-g_{
    {\bf u}}\rangle = a\frac{ \sqrt{n}\langle Q_1,g_{\bs}-g_{
    {\bf u}}\rangle}{\sqrt{F_{a,b}(Q_1, Q_2)}}+ b\frac{ \sqrt{n}\langle Q_2,g_{\bs}-g_{
    {\bf u}}\rangle}{\sqrt{F_{a,b}(Q_1, Q_2)}}.$$
  It thus suffices to show that for all $Q\in \mathcal{Q}$ we have
  $$\sqrt{n}\langle Q,g_{\bs}-g_{
    {\bf u}}\rangle\mid{\bf c}\stackrel{d}{\to}N(0,\upsilon).$$
   Setting  
    $$q_i := \sqrt{n} \int_{I_i}Q(x)dx\qquad\text{ where }\qquad I_i := \Big(\frac{i-1}{n}, \frac{i}{n}\Big],$$
    the above is equivalent to the claim
$$\sum_{i=1}^nq_i(\sigma_i-u_i)\mid{\bf c}\stackrel{d}{\to}N(0,\upsilon).$$
This follows by invoking part (b) of \cref{lem:ER}, for which we need to verify that ${\bf q}$ satisfies \eqref{eq:contrastcon}, part (b) of \cref{lem:ER} implies the desired convergence. We check each parts of  \eqref{eq:contrastcon} below.
    \begin{itemize}[itemsep=1em]
        \item 
       $\sumin q_i^2 \to 1$
 \\      
       
       Let $U$ be a random variable which is uniform on $(0,1]$, and let $\mathcal{F}_n$ be a $\sigma$-field defined by $\mathcal{F}_n := \sigma(\{U \in I_i: i \in [n] \})$. Then, setting $$Y_{n} := \EE[Q(U) \mid \mathcal{F}_n]=1\{U\in I_i\}n\int_{I_i}Q(x)dx=1\{U\in I_i\}\sqrt{n}q_i,$$  by Lévy's zero-one law (e.g. Theorem 4.6.8 in \cite{durrett2019probability}) we get $Y_n \stackrel{L^2}{\to} Q(U)$, and consequently 
        $$\sum_{i=1}^nq_i^2=\EE Y_n^2 \to \EE Q(U)^2 = \int_0^1 Q(x)^2 dx = 1,$$
        as desired. Note that while \eqref{eq:contrastcon} requires the equality $\sumin q_i^2 = 1$, this can be achieved by re-normalizing $\bq$ without changing the limiting variance.

        \item $\sumin q_i =0$
        \\

        This follows on noting that $\sum_{i=1}^nq_i=
        \sqrt{n} \int_0^1 Q(x) dx = 0$.

        \item $\|\bq\|_\infty \to 0$
        \\
        
        This follows on noting that
        $$\|\bq\|_\infty = \maxin \sqrt{n} \int_{I_i} Q(x) dx \le \maxin \sqrt{\int_{I_i} Q(x)^2 dx} \to 0,$$
        where the last limit follows on using DCT.
    \end{itemize}
\end{proof}

\begin{proof}[Proof of \cref{cor:reg}]
It is immediate to check that $\alpha_n= \red{\theta^2/d(n)}$ and hence $\sqrt{n}\alpha_n = \red{\theta^2 \sqrt{n}/d(n)} \to 0$. Also we have
$\|{\bf A}_n\|_\infty =|\theta|$, and so the SHT condition holds for $\theta\in (-1,1)$.
    \begin{enumerate}[(a).]
        \item In this case follows from \cref{cor:i.i.d. clt} by noting that $(\bq,\lambda_n)=(n^{-1/2}\mathbf{1},\red{\theta})$ forms an exact eigenvector-eigenvalue pair.

        \vspace{0.02in}
        
        \item In this case, taking the pair $(\bq,0)$ we have
        $$\sqrt{n\alpha_n}\lVert \A_n \bq \rVert \le \sqrt{\frac{n}{d(n)}}\left(\lVert \A_n - n^{-1}\mathbf{1}\mathbf{1}^{\top}\rVert\, +\, \frac{|\sum_{i=1}^n q_i|}{\sqrt{n}}\right)\lesssim \frac{\sqrt{n}}{d(n)}+\frac{|\sum_{i=1}^nq_i|}{\sqrt{d(n)}} = o(1).$$
    \end{enumerate}
\end{proof}

\begin{proof}[Proof of \cref{cor:graphon}]
    The result follows on invoking parts (a) and (b) of \cref{cor:i.i.d. clt}. We show below that the conditions of \cref{cor:i.i.d. clt} hold, and the error terms converge to $0$.

    \begin{itemize}
\item{\bf MHT condition \eqref{eq:mht}}
\\

Recall from \cref{lem:normorder} that $\|\A_n\|_4\le \|\A\|_\infty$, and so it suffices to show the stronger SHT condition \eqref{eq:sht}. To this effect, using Bernstein's inequality for bounded random variables (e.g. Theorem 2.8.4 in \cite{vershynin2018high}) alongside the fact that $$
   \EE[ \sumjn \G_n(i,j)\mid\mathbf{U}]= \sum_{j=1}^n\frac{f(U_i, U_j)}{n^\gamma} \le n^{1-\gamma},\quad \Var[\sumjn \G_n(i,j)\mid \mathbf{U}] \le n^{1-\gamma},$$ for any $\delta>0$ we have
    \begin{align}\label{eq:bernstein}
        \notag\PP(\|\G_n\|_\infty > (1+\delta)n^{1-\gamma}) &= \PP\Big(\maxin  \sumjn \G_n(i,j) > (1+\delta) n^{1-\gamma} \Big) \\
        &\le 2n\exp\left[ - \frac{\delta^2}{2(1+\delta/3)}n^{1-\gamma}\right] = o(1).
    \end{align}
    Consequently, we have
    \begin{align*}
        \|\A_n\|_\infty=\frac{|\theta|}{n^{1-\gamma}}\|\G_n\|_\infty\le |\theta|,
    \end{align*}
    where the last inequality holds with high probability using \eqref{eq:bernstein}. Thus we have verified the SHT condition \eqref{eq:sht} if $|\theta|<1$.
\\

    \item{\bf SMF condition \eqref{eq:smf}}
    \\

    Using the fact that $\G_n(i,j)\in \{0,1\}$, we have
    \begin{align*}
        \an = \maxin \sumjn \A_n(i,j)^2 = \maxin \frac{\theta^2}{n^{2-2\gamma}} \sumjn \G_n(i,j) \le  \frac{\theta^2\red{(1+\delta)}}{n^{1-\gamma}},
    \end{align*}
    where the last step holds with high probability using \eqref{eq:bernstein}. Since $\gamma<1/2$, we have shown that $\sqrt{n}\alpha_n\stackrel{P}{\to}0$, and so the SMF condition \eqref{eq:smf} holds.
    \\
    
   \item
    Next, we check each part of \eqref{eq:standassn}, where we take $q_i := \frac{Q(U_i)}{\sqrt{n}}$. %
    \begin{itemize}[itemsep=1em]
        \item $\sumin q_i^2 \to 1$ \\

        \noindent This follows as $\sumin q_i^2 = \frac{1}{n} \sumin Q(U_i)^2 \xp \EE Q(U_1)^2 =1.$ While \eqref{eq:standassn} requires the equality $\sumin q_i^2 = 1$, this can be achieved by re-normalizing $\bq$ without changing the limiting variance.
        
        \item $\EE\|\A_n \bq - \tilde{\lambda} \bq\|^2 =O\Big(\frac{1}{n^{1-\gamma}}\Big)$ for $\tilde{\lambda} := \theta \lambda$ \\
        
        \noindent A direct expansion gives
        $$\EE\|\A_n\bq - \tilde{\lambda} \bq\|^2 = \EE\Big[\sumin (\sumjn \A_n(i,j) q_j)^2\Big] -2 \tilde{\lambda} \EE \Big[\sum_{i,j} \A_n(i,j)q_iq_j \Big] + \tilde{\lambda}^2 \EE \|\bq\|^2.$$ 
        We simplify each summand. For the third term we have $\tilde{\lambda}^2 \EE \|\bq\|^2= \tilde{\lambda}^2$.
        Proceeding to estimate the second term, direct computations using tower property and the assumption that $Q$ is an eigenfunction of \red{$f$} gives
        \begin{align*}
            \EE\Big[\sum_{i,j} \A_n(i,j)q_iq_j \Big] &= \frac{\theta}{n^{2}}\sum_{i\neq j} \EE[f(U_i, U_j) Q(U_i) Q(U_j)] = \frac{\tilde{\lambda} (n-1)}{n}.
        \end{align*}
        Hence, the middle term is $2 \tilde{\lambda}^2 + O\Big(\frac{1}{n}\Big).$
        Finally, to compute the first term, write
        \begin{align*}
            \EE\Big[(\sumjn \A_n(i,j) q_j)^2 \mid \mathbf{U}\Big] &= \Big[\EE [\sumjn \A_n(i,j) q_j \mid \mathbf{U}] \Big]^2 + \Var\Big[\sumjn \A_n(i,j) q_j \mid \mathbf{U} \Big] \\
            &= \frac{\theta^2}{n^2} \Big[\sum_{j \neq i} f(U_i,U_j) q_j\Big]^2+ O\Big(\frac{\red{\sumjn q_j^2}}{n^{2-\gamma}}\Big).
        \end{align*}
        Taking a conditional expectation given $U_i$ on both sides we get         \begin{align*}
            &\EE\Big[(\sumjn \A_n(i,j) q_j)^2 \mid U_{i} \Big] \\
            =& \frac{\theta^2}{n^2} \Big[\EE \big[\sum_{j \neq i} f(U_i,U_j) q_j \mid U_i\big] \Big]^2 + \frac{\theta^2}{n^2} \Var\Big[\sum_{j \neq i} f(U_i,U_j) q_j \mid U_i \Big] + O\Big(\frac{\red{\sumjn q_j^2}}{n^{2-\gamma}}\Big) \\
            =& \frac{\red{\theta^2}}{n^2} \Big[\frac{n-1}{\sqrt{n}} \lambda Q(U_i) \Big]^2 +  O\Big(\frac{1}{n^2} \Big) + O\Big(\frac{\red{\sumjn Q(U_j)^2}}{n^{3-\gamma}}\Big).
        \end{align*}
        Here, the last line again uses the fact that $Q$ is the eigenfunction of \red{$f$}, along with the fact that $q_j=n^{-1/2} Q(U_j)$. Finally, by taking the expectation over $\mathbf{U}$ and summing over $i$, we have
        $$\EE\Big[\sumin (\sumjn \A_n(i,j) q_j)^2 \Big] = \frac{\theta^2\lambda^2(n-1)^2}{n^3}\EE\Big[\sum_{i=1}^nQ^2(U_i)\Big]+O\Big(\frac{1}{n^{1-\gamma}}\Big)=\tilde{\lambda}^2+O\Big(\frac{1}{n^{1-\gamma}}\Big).$$%
       Combining the above bounds we have $\EE\|\A_n\bq - \tilde{\lambda} \bq\|^2 = O\Big(\frac{1}{n^{1-\gamma}}\Big)$, as claimed.
        
        \item $\|\bq\|_\infty \to 0$ \\
        
        \noindent It follows from the usual law of large numbers that for i.i.d. random variables $Y_i$ with finite mean, we have $\frac{\maxin Y_i}{n} \xrightarrow{P} 0$. Noting that $\EE Q^2(U_i)=1$, this gives $$\|\bq\|_\infty^2=\frac{1}{n}\max_{1\le i\le n}|Q^2(U_i)| \stackrel{P}{\to}0.$$
    \end{itemize}    

   Having verified all its conditions, the CLTs now follow from \cref{cor:i.i.d. clt}. The additional condition in part (b) of \cref{cor:i.i.d. clt} also follows from the above bounds, as
    $$n \an \|\A_n\bq- \red{\tilde{\lambda}} \bq\|^2 = O_P\left(\frac{n}{n^{2(1-\gamma)}} \right) = o_P(1),$$
    where the last step in the above display uses the fact that $\gamma<\frac{1}{2}$.
    Note that while we have established all bounds while marginalizing out $\mathbf{U}$, the analogous claims conditional on $\mathbf{U}$ follow from this.
    \end{itemize}
\end{proof}

To prove \cref{lem:hopfield condition checking}, we require the following lemmas to check $\sqrt{n} \an = o_P(1)$ as well as the MHT condition. Note that unlike the previous examples in the context of graphs, bounding the norm $\|\A_n\|_{4}$ by $\|\A_n\|_{\infty}$ is suboptimal for spin glasses. We postpone the proofs of these lemmas to \cref{sec:pfauxlem}. %

\begin{lemma}\label{lem:L_n norm}
    Let $\mathbf{Z}$ be a $N \times n$ random matrix whose entries are independent sub-Gaussian random variables with mean zero and variance $1$, and sub-Gaussian norm bounded by a constant. Let $\M_n := \frac{1}{N} \mathbf{Z}^\top \mathbf{Z}$, and let $\G_n$ be the adjacency matrix of a graph.
    Suppose $N \ge n$.
    For $\A_n(i,j) := \theta\M_n(i,j)\G_n(i,j)$ and $r \in [2, \infty]$, the following bound holds:
    \begin{equation}\label{eq:l_n norm bound}
        \|\A_n\|_{r} = O_P\Big(\frac{n^{1-1/r}}{\sqrt{N}}\Big).
    \end{equation}
    While not necessary for our analysis, we also have $\|\A_n\|_{r} = O_P\Big(\frac{n^{1/r}}{\sqrt{N}}\Big)$ for $r \in [1,2)$.
\end{lemma}

\begin{remark}
    This result resembles Chevet's inequality \citep{chevet1978series, latala2024operator}, which provides similar $r \to r$ norm bounds for the case when $\A_n$ has i.i.d. entries. Note that taking $\G_n(i,j) = I(i \neq j)$ in \cref{lem:L_n norm} implies that $\A_n$ is the off-diagonal part of the sample covariance matrix, and one can easily conclude that
    $$\|\M_n - \I_n\|_{r} = O_P\Big(\frac{n^{1-1/r}}{\sqrt{N}}\Big)$$
    for $r \ge 2$ and $N \ge n$.
    We are not aware of analogous results for covariance matrices, and this lemma may be of independent interest.
\end{remark}

\begin{lemma}\label{lem:alpha_n}
    Define $\mathbf{A}_n$ as in \cref{lem:L_n norm}. Under the asymptotic scaling $N \gg n^{3/2}$, we have $\maxin \sum_{j=1}^n \mathbf{A}_n(i,j)^2 = o_P(1/\sqrt{n}).$ %
\end{lemma}

\begin{proof}[Proof of \cref{lem:hopfield condition checking}]
    The simple expressions of the variance terms in~\cref{cor:i.i.d. clt} follow by noting that our assumptions imply $\psi(c) = \log\cosh(c)$, and the fact that we take $\lambda=0$.

    \vspace{0.05in}

    \noindent %
    By \cref{lem:alpha_n},
    we have $\sqrt{n} \an = o_P(1).$
    Also, by \cref{lem:L_n norm} with $r=4$, the MHT condition \eqref{eq:mht} holds for all $\theta$. Finally, by \cref{lem:L_n norm} with $r = 2$,
    we have 
     $\lVert\A_n \rVert = O_P\left( \sqrt{\frac{n}{N}} \right),$ and so
     $$\sqrt{n\alpha_n}\lVert \A_n \bq \rVert =o_P(n^{1/4})\lVert \A_n \bq \rVert \le o_P(n^{1/4})\lVert \A_n\rVert = o_P(1),$$
     where the last estimate uses the fact that $N\gg n^{3/2}$.
\end{proof}

\section{Proof of main lemmas}\label{sec:pfmainlem}

\begin{proof}[Proof of \cref{lem:sum L2 bound}]
\red{Without the loss of generality, assume that the function $g$ is bounded by $1$, i.e. $\|g\|_\infty \le 1$.}

\vspace{2mm}
	\emph{Proof of (a).} As in the proof of \cref{thm:clt}, we generate $\bs'$ as follows: let $I$ be a randomly sampled index from $\{1,2,\ldots ,n\}$. Given $I=i$, replace $\sigma_i$ with an independently generated $\sigma'_i$ drawn from the conditional distribution of $\sigma_i$ given $\sigma_j,\ j\neq i$. Recall the definition of $\bb^{(g)}$ from \cref{def:conex} and let $$F(\bs, \bs') := \sumin \gamma_i(g(\sigma_i) - g(\sigma_i')) = \gamma_I(g(\sigma_I) - g(\sigma_I')),$$
 
 $$f(\bs) := \EE [ F(\bs, \bs') | \bs] = \frac{1}{\red{n}} \sumin \gamma_i (g(\sigma_i) - b_i^{(g)}).$$ Define $\Delta(\bs):= \EE\left( |f(\bs) - f(\bs')| |F(\bs, \bs')| \big | \bs \right).$ 
	By \cite[Theorem 1.5(ii)]{chatterjee2006stein}, it suffices to show 
 \begin{equation}\label{eq:toshow1}
 \Delta(\bs) \lesssim \frac{\|\boldsymbol{\gamma}\|^2}{n^2}
 \end{equation}
 where the hidden constant depends only on $h$.
 \vspace{0.1in}
 
	Towards this direction, we first recall the notation in \cref{def:expfam} and define the following mapping:
 $$\theta \mapsto \int g(x)\exp(\theta x-\psi_i(\theta))\,d\mu_{i}(x)\ =:\ \mathcal{G}_i(\theta).$$
{By taking derivatives, we can see that
$$\mathcal{G}_i'(\theta) = \Cov_{\mu_{i,\theta}} \left(X, g(X) \right), ~ \mathcal{G}_i''(\theta) = \Cov_{\mu_{i,\theta}} \left((X- \psi_i'(\theta))^2, g(X) \right),$$
and $\mathcal{G}_i', \mathcal{G}_i''$ are uniformly bounded by a constant not depending on $i$.
}
 In the same vein as \eqref{eq:locfield}, we write 
 $$m_i':=\sumjn \A_n(i,j)\sigma'_j=m_i+\A_n(i,I)(\sigma'_I-\sigma_I).$$
 A direct computation also yields that $b_i^{(g)}=\EE[g(\sigma_i)|\sigma_j, j\ne i]=\mathcal{G}_i(m_i+c_i),$ and so 
 $$f(\bs')=\frac{1}{n}\sumin \gamma_i(g(\sigma_i')-\mathcal{G}_i( m_i'+c_i)).$$
 Combining the above observations with a standard Taylor series expansion, there exist reals $\{\xi_{i,j}\}_{1\le i,j\le n}$, such that:
	\begin{align}\label{eq:difftay}
		&\;\;\;\;\mathcal{G}_i( m_i'+c_i)-\mathcal{G}_i( m_i+c_i) \nonumber \\ &=  (m_i' - m_i) \mathcal{G}_i'( m_i + c_i) + \frac{(m_i - m_i')^2}{2} \mathcal{G}_i''( {\xi_{i,I}}+c_i)\nonumber \\
		&=  \A_n(i,I) (\sigma_I' - \sigma_I)\mathcal{G}_i'( m_i + c_i) + \frac{1}{2}\A_n(i,I)^2 (\sigma_I - \sigma_I')^2 \mathcal{G}_i''({\xi_{i,I}} + c_i).
	\end{align}
	As a result, we have
	\begin{align*}
		&|f(\bs) - f(\bs')|  \\
		=& \frac{1}{n} \left|\gamma_I(g(\sigma_I) - g(\sigma_I')) - \sumin \gamma_i (\mathcal{G}_i( m_i+c_i)-\mathcal{G}_i( m_i'+c_i)) \right| \\
		\le& \frac{1}{n} |\gamma_I|\cdot|g(\sigma_I) - g(\sigma_I')| + \frac{1}{n}\left|  \sumin \gamma_i \A_n(i,I) (\sigma_I-\sigma'_I)\mathcal{G}_i'( m_i + c_i)\right| \\ &\quad + \frac{1}{n}\left|\frac{1}{2} \sumin \gamma_i \A_n(i,I)^2 \mathcal{G}_i''( \xi_{i,I} + c_i) (\sigma_I- \sigma_I')^2 \right|.
	\end{align*}
	Recall that all the measures $\{\mu_i\}_{1\le i\le n}$ are supported on $[-1,1]$, and $g(\cdot)$ takes values in $[-1,1]$ as well. Define a matrix $\mathbf{C}_n$ such that $\mathbf{C}_n(i,j):= \A_n(i,j)\mathcal{G}_i'( m_i+c_i)$. Then, by our assumption that $\limsup_{n \to \infty} \lVert \A_n \rVert < \infty$, we have:
 \begin{align}\label{eq:reflat}
  \sumin \gamma_i \A_n(i,I)\mathcal{G}_i'( m_i+c_i)=(\C_n^{\top}\boldsymbol{\gamma})_{I},\quad \limsup\limits_{n\to\infty} \, \lVert \C_n^{\top}\C_n\rVert\le \limsup\limits_{n\to\infty} \,\lVert \A_n\rVert^2 <\infty.
 \end{align}
 To simplify notation in the following sequence of displays, let us use $\EE_{\bs}$ to denote the \textit{conditional} expectation given $\bs$. Therefore, 
	\begin{align*}
	\Delta(\bs) &= \EE_{\bs}\left( \left| \gamma_I (g(\sigma_I) - g(\sigma_I')) \right| \left| f(\bs) - f(\bs') \right| \right) \\ & \le \frac{1}{n}\EE_{\bs}[\gamma_I^2(g(\sigma_I)-g(\sigma'_I))^2] + \frac{1}{n}\EE_{\bs}\left[|\gamma_I|\cdot|(g(\sigma_I)-g(\sigma'_I))(\sigma_I-\sigma'_I)(\C_n^{\top}\bgamma)_I|\right]\\ & \quad + \frac{1}{n}\EE_{\bs}\left|\frac{1}{2}\sumin \gamma_i\gamma_I \A_n(i,I)^2(\sigma_I-\sigma'_I)^2 (g(\sigma_I)-g(\sigma_I'))\mathcal{G}_i''( \xi_{i,I}+c_i)\right|\\
		&\lesssim \frac{\lVert \boldsymbol{\gamma}\rVert^2}{n^2} + \frac{1}{n^2}\lVert \boldsymbol{\gamma}\rVert \lVert \C_n^{\top}\boldsymbol{\gamma}\rVert + \frac{1}{n^2} \sum_{i,j} |\gamma_i| |\gamma_j| \A_n(i,j)^2  \\ & \lesssim \frac{\lVert \boldsymbol{\gamma}\rVert^2}{n^2} + \frac{\alpha_n \lVert \boldsymbol{\gamma}\rVert^2}{n^2} \lesssim \frac{\lVert \boldsymbol{\gamma}\rVert^2}{n^2}.
	\end{align*}
	The first and second inequalities in the above sequence follow directly from \eqref{eq:reflat}. By defining $\mathbf{B}_n(i,j):=\A_n^2(i,j)$, the final inequality follows from \red{using \cref{lem:normorder} to write}
    \begin{align}\label{eq:b_n matrix}
        \lVert \mathbf{B}_n\rVert\le \red{\lVert\mathbf{B}_n\rVert_\infty} =  \max_{i=1}^n\sumjn \mathbf{B}_n(i,j)=\alpha_n
    \end{align}
    (recall the definition of $\an$ from \eqref{eq:rowcontrol}). This establishes \eqref{eq:toshow1} and  completes the proof of the exponential tail bound.
 
 \vspace{0.1in}
	
	The moment bounds follow from the exponential tail bound by noting that, for some constant $a>0$, the following holds:
	\begin{align*}
		\EE & \left|\sumin \gamma_i(g(\sigma_i) - \mathcal{G}_i( m_i+c_i)\right|^r = \int_0^{\infty} r t^{r-1} \PP\left(\left|\sumin \gamma_i(g(\sigma_i) - \mathcal{G}_i( m_i+c_i))\right| > t\right) d t \\
		& \leq 2r\int_0^{\infty} t^{r-1} e^{-\frac{a t^2}{\|\bgamma\|^2}} \, dt  = \frac{r\lVert \bgamma\rVert^r}{a^{\frac{r}{2}}} \int_0^{\infty} t^{\frac{r-2}{2}} e^{-t}\,dt \lesssim  \frac{\|\bgamma\|^r}{(2a e)^{\frac{r}{2}}} r^{\frac{r+1}{2}}.
	\end{align*}
        The final inequality uses standard upper bounds on the Gamma function, see e.g.,~\cite[Theorem 1.1]{nemes2010new}. We note here that the hidden constants do not depend on $r$. This completes the proof.
 \vspace{0.05in}
 
	\emph{Proof of (b).}
	We proceed in a similar fashion as in part (a). Define 
\begin{align*}
    m_j^i := \sum_{k \neq i} \A_n(j,k) \sigma_k.
\end{align*}
Set $b^{(g)}_j=\mathcal{G}_j( m_j+c_j)$ (from \cref{def:conex}) and set  $b^{(g),i}_j:=\mathcal{G}_j(m_j^i+c_j)$ for $j\neq i$. 
	\begin{align*}
		&\;\;\;\;\left|\sumin \gamma_i \big(g(\sigma_i) - b^{(g)}_i\big) b^{(g)}_i\right|^2 
		\\ &= \sumin \gamma_i^2 \big(g(\sigma_i) - b^{(g)}_i\big)^2 \big(b^{(g)}_i\big)^2 + \sum_{i \neq j} \gamma_i \gamma_j \big(g(\sigma_i) - b^{(g)}_i\big) b^{(g)}_i \big(g(\sigma_j) - b^{(g)}_j\big) b^{(g)}_j \\
		&= \sumin \gamma_i^2 \big(g(\sigma_i) - b^{(g)}_i\big)^2 \big(b^{(g)}_i\big)^2  + \sum_{i \neq j} \gamma_i \gamma_j \big(g(\sigma_i) - b^{(g)}_i\big) b^{(g)}_i \Big[ \big(g(\sigma_j) - b^{(g),i}_j\big) b^{(g),i}_j \\ & \quad \quad + \big(g(\sigma_j) - 2 b^{(g)}_j\big) \big(b^{(g)}_j - b^{(g),i}_j\big) - \big(b^{(g)}_j - b^{(g),i}_j\big)^2 \Big].
	\end{align*}
	Since $\EE\big(g(\sigma_i) - b^{(g)}_i \mid \sigma_j,\ j\neq i\big) = 0$, and $b^{(g)}_i, b^{(g),i}_j$ are both measurable with respect to $(\sigma_k,\ k\neq i)$, the term $$\EE\left[\big(g(\sigma_i) - b^{(g)}_i\big) b^{(g)}_i \big(g(\sigma_j) - b^{(g),i}_j\big) b^{(g),i}_j\right]=0,\quad\quad \mbox{for all}\ j\neq i.$$ 
Once again, we set $\mathbf{B}_n(i,j) := \A_n^2(i,j)$. Using the above displays, we then get:
		\begin{align}
			&\;\;\;\;\EE \left|\sumin \gamma_i (g(\sigma_i) - b^{(g)}_i) b^{(g)}_i\right|^2 \notag \\ &\lesssim \sumin \gamma_i^2 + \EE\bigg|\sum_{i \neq j} \gamma_i \gamma_j (g(\sigma_i) - b^{(g)}_i) b^{(g)}_i \Big[\big(g(\sigma_j) - 2 b^{(g)}_j\big) \big(b^{(g)}_j - b^{(g),i}_j\big)|\Big]\bigg| \label{eq:moment RHS} \\
            &\qquad + \EE\bigg|\sum_{i \neq j} \gamma_i \gamma_j (g(\sigma_i) - b^{(g)}_i) b^{(g)}_i \big(b^{(g)}_j - b^{(g),i}_j\big)^2 \bigg|. \notag
		\end{align}
 By a similar Taylor expansion as in \eqref{eq:difftay}, we note that for some $\xi_{i,j}\in\R$:
\begin{align}\label{eq:two bounds for b difference}
    b^{(g),i}_j-b^{(g)}_j = -\A_n(i,j)\sigma_i\mathcal{G}_j'(m_j+c_j) \red{+} \frac{1}{2}\A_n(i,j)^2\sigma_i^2\mathcal{G}_j''(\xi_{i,j}+c_j), \quad \red{|b^{(g),i}_j-b^{(g)}_j| \lesssim |\A_n(i,j)|.}
\end{align}
\red{By using each bound in \eqref{eq:two bounds for b difference} for the second and third summand in \eqref{eq:moment RHS}}, respectively, we have
\begin{align*}
     &\;\;\;\;\EE \left|\sumin \gamma_i (g(\sigma_i) - b^{(g)}_i) b^{(g)}_i\right|^2 \\ &\lesssim \sumin \gamma_i^2 +  \EE \bigg|\sum_{i,j} \left[\gamma_i \big(g(\sigma_i) - b^{(g)}_i\big) b^{(g)}_i \sigma_i \right] \A_n(i,j) \left[\gamma_j \big(g(\sigma_j) - 2 b^{(g)}_j\big) \mathcal{G}_j'( m_j + c_j) \right]\bigg| \\
     &\quad \red{+ \frac{1}{2} \EE\bigg|\sum_{i \neq j} \gamma_i \gamma_j (g(\sigma_i) - b^{(g)}_i) b^{(g)}_i \big(g(\sigma_j) - 2 b^{(g)}_j\big) \A_n(i,j)^2\sigma_i^2\mathcal{G}_j''(\xi_{i,j}+c_j) \bigg|} \\
     &\quad \red{+ \EE \bigg[\sum_{\i \neq j} |\gamma_i \gamma_j (g(\sigma_i) - b^{(g)}_i) b^{(g)}_i|  \mathbf{A}_n(i,j)^2 \bigg]} \\
	&\lesssim \sumin \gamma_i^2 +  \EE \bigg|\sum_{i,j} \left[\gamma_i \big(g(\sigma_i) - b^{(g)}_i\big) b^{(g)}_i \sigma_i \right] \A_n(i,j) \left[\gamma_j \big(g(\sigma_j) - 2 b^{(g)}_j\big) \mathcal{G}_j'( m_j + c_j) \right]\bigg| \\
    &\quad + \sum_{i,j} |\gamma_i| \mathbf{B}_n(i,j) |\gamma_j| \, \lesssim \, \sumin \gamma_i^2.
 \end{align*}
 In the last inequality, we bounded (a) the second term by viewing it as a quadratic form, the (b) third term by also viewing it as a quadratic form alongside the operator norm bound of $\mathbf{B}_n$ in \eqref{eq:b_n matrix}.
\end{proof}

\noindent Next, we move on to the proof of \cref{lem:m-n contraction}.  We first need the following matrix lemma whose proof we defer.

		\begin{lemma}\label{lem:matrix norm} 
  Let $\|\cdot\|_q$ be the $q \to q$ operator norm, with $q \in [2, \infty]$. Also, let $\mathbf{C}_n$ be a $n \times n$ matrix. Suppose $\|\mathbf{C}_n\|_q \le \theta < 1$. Then, $(\I_n - \mathbf{C}_n)^{-1}$ exists and satisfies $\|(\I_n-\mathbf{C}_n)^{-1}\|_q \le \big(1 - \theta\big)^{-1}$. 
           \end{lemma} 
\begin{proof}[Proof of Lemma \ref{lem:m-n contraction}]
		
Suppose $\lVert \A_n \rVert_q < 1$, for some $q \in [2,\infty]$.
		Let $$\mathfrak{f}_i := \sumjn \A_n(i,j) (\sigma_j - b_j) = m_i - \sumjn \A_n(i,j) \psi_j'( m_j + c_j).$$ Lemma \ref{lem:sum L2 bound}(a) with $g(x):=x, \gamma_j = \A_n(i,j)$ shows that for all $i$, and $q\in [2,\infty)$, we have 
  \begin{equation}\label{eq:tailbdcall}
  \EE |\mathfrak{f}_i|^q\lesssim \left(\frac{\alpha_n}{2ae}\right)^{\frac{q}{2}} q^{\frac{q+1}{2}},
  \end{equation}
  where the hidden constant is free of $q$. By a Taylor expansion, we can write
		\begin{align*}
			m_i - s_i &= \sumjn \A_n(i,j) (\psi_j'( m_j + c_j) - \psi_j'( s_j + c_j)) + \mathfrak{f}_i \\
            &=  \sumjn \A_n(i,j) (m_j - s_j) \psi_j''( \xi_j + c_j) + \mathfrak{f}_i.
		\end{align*}
		Hence, $\mathbf{m} - \mathbf{s} = \C_n(\mathbf{m}-\mathbf{s}) + \boldsymbol{\mathfrak{f}}$, where $\mathbf{C}_n(i,j) :=  \A_n(i,j) \psi_j''( \xi_j + c_j)$ and $\boldsymbol{\mathfrak{f}}:=(\mathfrak{f}_1,\ldots ,\mathfrak{f}_n)^{\red{\top}}$. Note  $$\|\C_n\|_{q} \le \lVert \A_n\rVert_q < 1.$$ Then, \cref{lem:matrix norm} implies $\|(\I_n-\C_n)^{-1}\|_{q} < \infty$. Thus,
\begin{equation}\label{eq:contractioncall}
  \|\mathbf{m}-\mathbf{s}\|_q \le \|(\I_n-\C_n)^{-1}\|_{q} \|\boldsymbol{\mathfrak{f}}\|_q.
\end{equation}  
  Note that this bound is deterministic.
  
  \vspace{0.05in} 

 \noindent \emph{Proofs of (a) and (b).} Part (a) now follows by taking $q = 2$ in \eqref{eq:tailbdcall} and \eqref{eq:contractioncall}. Similarly, part (b) follows by taking $q = 4$ in \eqref{eq:tailbdcall} and \eqref{eq:contractioncall}.
   
   \vspace{0.05in}
   
\noindent \emph{Proof of (c).} 
  Under the SHT condition \eqref{eq:sht}, we have $\|\C_n\|_\infty \le \|\A_n\|_\infty \le \rho$.
  By using \cref{lem:normorder} and \cref{lem:matrix norm}, for any given $r\in [2,\infty)$, we have 
  $$\red{\|(\I_n-\C_n)^{-1}\|_r \le \|(\I_n-\C_n)^{-1}\|_{\infty}\le \frac{1}{1-\rho}}.$$
  Then, taking $q = r$ in \eqref{eq:contractioncall} gives
		$$\sumin |m_i-s_i|^r \le \frac{1}{(1-\rho)^r} \sumin |\mathfrak{f}_i|^r.$$
	The conclusion now follows by invoking \eqref{eq:tailbdcall} with $q=r$.
\end{proof}

\begin{proof}[Proof of \cref{lem:optimizer u l2 concentration}]
   By the definition of $\mathbf{s}$ in \cref{lem:uniqueness of optimizers}, we have
    \begin{align*}
    		s_i = \sumjn \A_n(i,j) \psi_j'( s_j + c_j) 
    		= {\sumjn \A_n(i,j) \psi_j'(c_j)} +  \sumjn \A_n(i,j) s_j \psi_j''( \xi_j + c_j)
    \end{align*}
    for some constants $\xi_j$. Let $\mathbf{D}_n$ be a $n \times n$ matrix with $$\mathbf{D}_n(i,j) :=  \A_n(i,j) \psi_j''( \xi_j + c_j).$$
    Then, we have $(\I_n-\mathbf{D}_n) \mathbf{s} = \boldsymbol{\nu}.$

\vspace{0.05in}

    \noindent \emph{Proof of (a).} Under the WHT condition \eqref{eq:wht} we have $\lVert \mathbf{D}_n \rVert \le \lVert \A_n \rVert \le \rho<1$. Hence, by applying \cref{lem:matrix norm}, we have
    $$\lVert  \mathbf{s} \rVert \lesssim \lVert \boldsymbol{\nu} \rVert.$$

\vspace{0.05in}

    \noindent \emph{Proof of (b), (c).} Similarly, under the MHT \eqref{eq:mht} (or SHT \eqref{eq:sht})  condition we have $\lVert \mathbf{D}_n \rVert_r \le \lVert \A_n\rVert_r \le \rho < 1$ with $r = 4$ (or $r=\infty$), and we have
    $$\lVert  \mathbf{s} \rVert_r \lesssim \lVert \boldsymbol{\nu} \rVert_r,$$	
    by applying \cref{lem:matrix norm}.  This completes the proof.
\end{proof}

\section{Appendix: Proof of technical lemmas}\label{sec:pfauxlem}

Here, we present the proofs of the auxiliary lemmas that were used in previous proofs, which include high-probability bounds for random matrices (Lemmas \ref{lem:normorder}, \ref{lem:ksbound}, \ref{lem:L_n norm}, \ref{lem:alpha_n}, and \ref{lem:matrix norm}).

\begin{proof}[Proof of \cref{lem:normorder}]\label{pf:normorder}
    	To show $\lVert\M_n\rVert\le \lVert\M_n\rVert_{r}$, let $\lambda(\M_n)$ be the $(2,2)$-operator norm of $\M_n$ (i.e. maximum absolute eigenvalue) with an eigenvector $\mathbf{v}$, i.e., $\M_n \mathbf{v} = \lambda(\M_n) \mathbf{v}$. Then,
	$$\lVert\M_n\rVert_r^r = \sup_{\mathbf{x} \neq 0} \frac{\sumin |(\M_n \mathbf{x})_i|^r}{\sumin |x_i|^r}\ge \frac{\sumin |(\M_n \mathbf{v})_i|^r}{\sumin |v_i|^r} = \lambda(\M_n)^r,$$
	so $\lVert \M_n\rVert_r \ge \lVert \M_n\rVert$.

 \vspace{0.05in}
 
	To show $\lVert \M_n\rVert_r\le \lVert \M_n\rVert_{\infty}$, 
	without the loss of generality, suppose that $\lVert \M_n\rVert_{\infty} = 1$. Then,
	\begin{align*}
		\lVert \M_n\rVert_r^r &= \max_{\sumin |x_i|^r = 1} \sumjn \left|\sumin \M_n(i,j) x_i\right|^r  \le \max_{\sumin |x_i|^r = 1} \sumjn \left(\sumin |\M_n(i,j)| |x_i|\right)^r
	\end{align*}
	and it suffices to show that 
	$$\sumjn \left(\sumin |\M_n(i,j)| |x_i|\right)^r \le \sumin |x_i|^r.$$
	By Jensen's inequality, we then have
	$$\left(\sumin |\M_n(i,j)| |x_i|\right)^r \le \sumin |\M_n(i,j)| |x_i|^r$$
	for all $j$. Hence, 
 \begin{align*}
 \sumjn \left(\sumin |\M_n(i,j)| |x_i|\right)^r \le \sumjn \sumin |\M_n(i,j)| |x_i|^r \le \sumin |x_i|^r.
 \end{align*}
 This completes the proof.
\end{proof}
\begin{proof}[Proof of \cref{lem:ksbound}]
For simplicity, set $W \equiv \mcn(0,\tau)$. Then, we have
\begin{align*}
    d_{KS}(aX+b,W) & \le d_{KS}(aX+b, aW+b) + d_{KS}(aW+b, W) \\
    &= d_{KS}(X, W) + d_{KS}(aW+b, W) \\
    &\le d_{KS}(X,W) + C_K\Big( |a^2-1|+ |b|\Big).
\end{align*}
Here, the second line follows from the definition of KS distance, and the third from the Lipschitz property of the KS distance between two Normals (e.g. Theorem 1.3 in \cite{devroye2018total}), and $C_K$ is some constant depending only on $K$.
\end{proof}

\begin{proof}[Proof of \cref{lem:L_n norm}]
    \red{For notational simplicity, assume that the inverse-temperature parameter is fixed at $\theta = 1$.}
    Fix $r \in [2, \infty]$. %
    Recall that the $r \to r$ operator norm can be written as
    \begin{equation}\label{eq:dual formula norm}
        \|\A_n\|_{r} = \sup_{\|\mathbf{v}\|_r=1, ~ \|\mathbf{w}\|_{r^\star}=1} \bv^\top \A_n \bw,
    \end{equation}
    where $r^\star \in [1,2]$ is the H\"older conjugate of $r$ that satisfies $\frac{1}{r} + \frac{1}{r^\star} = 1$ (e.g. see pg 2 of \cite{latala2024operator}). Let $S_r^n := \{\bv \in \mathbb{R}^n: \|\bv\|_r = 1\}$ be the $\ell^r$ sphere in $\mathbb{R}^n$. Also, for any $\varepsilon>0$, let $\mathcal{N}(S_r^n, \varepsilon) \red{\subset S_r^n}$ be an \red{$\varepsilon$-net} of $S_r^n$, with distance $d(\bv,\bv') = \|\mathbf{v} - \mathbf{v}'\|_r$. Then, repeating the standard volumetric argument (e.g. see Proposition 4.2.12 in \citep{vershynin2018high}) gives the upper bound
    \begin{align}\label{eq:net size}
        |\mathcal{N}(S_r^n, \varepsilon)| \le \Big( \frac{2}{\varepsilon}+1\Big)^n.
    \end{align}
    Indeed, as in the proof there, it suffices to upper bound the $\varepsilon$-packing number of $B_r^n(\mathbf{0}, 1):=\{{\bf v}\in \R^n:\|{\bf v}\|_r\le 1\}$, the closed unit $\ell^r$-ball in $\mathbb{R}^n$. Let $\mathcal{P}$ denote this value. Then, we can construct $\mathcal{P}$ closed disjoint balls $B_r^n(\mathbf{v}_i, \varepsilon/2)$ with centers $\mathbf{v}_i \in B_r^n(\mathbf{0}, 1)$ and radius $\varepsilon/2$. As these balls fit in $B_r^n(\mathbf{0}, 1+\varepsilon/2)$, comparing the volumes give
    $$\mathcal{P} \text{vol}\Big(B_r^n(\mathbf{v}, \varepsilon/2)\Big) \le \text{vol}\Big(B_r^n(\mathbf{0}, 1+\varepsilon/2)\Big),$$
    and \eqref{eq:net size} holds. Here, $\text{vol}(B)$ denotes the volume of a set $B\subset \mathbb{R}^n$.
    In the remainder of the proof, let us fix $\varepsilon = \frac{1}{4}$, which gives $|\mathcal{N}(S_r^n, \varepsilon)| \le 9^n$.
\\

    \textbf{Step 1: Netting.} We first claim that
    $$\|\A_n\|_{r} \le \frac{1}{1 - 2\varepsilon} \sup_{\substack{\bv \in \mathcal{N}(S_r^n, \varepsilon), \\ \bw \in \mathcal{N}(S_{r^\star}^n, \varepsilon)}} \bv^\top \A_n \bw = 2\sup_{\substack{\bv \in \mathcal{N}(S_r^n, \varepsilon), \\ \bw \in \mathcal{N}(S_{r^\star}^n, \varepsilon)}} \bv^\top \A_n \bw.$$
    Let $\bv_0 \in S_r^n, ~\bw_0 \in S_{r^\star}^n$ be vectors that attains the \red{supremum in \eqref{eq:dual formula norm}} so that $\|\A_n\|_{r} = \bv_0^\top \A_n \bw_0$. Take $\bv \in \mathcal{N}(S_r^n, \varepsilon), ~ \bw \in \mathcal{N}(S_{r^\star}^n, \varepsilon)$ such that $\|\bv - \bv_0\|_r < \varepsilon, ~ \|\bw - \bw_0\|_{r^\star} < \varepsilon$. Then, we can write
    \begin{align*}
        \|\A_n\|_{r} = \bv_0^\top \A_n \bw_0 &= \bv^\top \A_n \bw + (\bv_0-\bv)^\top \A_n \bw + \bv_0^\top \A_n (\bw_0 - \bw) \\
        &\le \bv^\top \A_n \bw + 2 \varepsilon \| \A_n\|_{r}.
    \end{align*}
    For the last line, we are using H\"older's inequality to bound
    \begin{align*}
        (\bv_0-\bv)^\top \A_n \bw &\le \|\A_n (\bv_0 - \bv)\|_r \|\bw\|_{r^\star} \le \varepsilon \|\A_n\|_{r}, \\
        \bv_0^\top \A_n (\bw_0- \bw) &\le \|\A_n \bv_0\|_r \|\bw_0 - \bw\|_{r^\star} \le \varepsilon \|\A_n\|_{r}.
    \end{align*}
    Then, we have the desired result by rearranging the terms and taking the supremum over $\bv \in \mathcal{N}(S_r^n, \varepsilon), \bw \in \mathcal{N}(S_{r^\star}^n, \varepsilon)$.
  \\
  
    \textbf{Step 2: Concentration.}
    We fix $\bv \in \mathcal{N}(S_r^n, \varepsilon), \bw \in \mathcal{N}(S_{r^\star}^n, \varepsilon)$, and compute the tail bound for $\PP( \bv^\top \A_n \bw > \delta)$. For notational simplicity, let us define a $n \times n$ matrix $\mathbf{B}_n$ (depending on $\G_n$) by setting $\mathbf{B}_n(i,j) = {\G_n(i,j)} v_i w_j I(i \neq j)$. In the remainder of the proof, all probabilities will be conditioned on the graph $\G_n$.
    For each $1\le k\le N$, let $$T_k = T_k^{\bv, \bw} := \sum_{i \neq j} Z_{k,i} Z_{k,j} {\G_n(i,j)} v_i w_j = \mathbf{Z}_k^\top \mathbf{B}_n \mathbf{Z}_k$$ be a quadratic form of $\mathbf{Z}_k = (Z_{k,1}, \ldots, Z_{k,n})^\top$. Then, we can write
    $$\bv^\top \A_n \bw = \frac{1}{N}\sum_{1 \le i \neq j \le n} \sum_{k=1}^N Z_{k,i} Z_{k,j} \G_n(i,j) v_i w_j = \frac{1}{N} \sum_{k=1}^N T_k.$$
    
    As $\{\mathbf{Z}_k\}_{1\le k\le N}$ are \red{i.i.d.}, so are $\{T_k\}_{1\le k\le N}$. %
    By a moment generating function based version of the Hanson-Wright inequality (e.g. see the last equation in Step 4 in the proof of Theorem 1.1 in \cite{rudelson2013Hanson-Wright}, where we take $A = \mathbf{B}_n$ and $S = \sum_{i\neq j} \mathbf{B}_n(i,j) Z_{k,i} Z_{k,j}$), we can write
    $$\EE \exp \Big[\theta T_k \Big] \le \exp\Big[K_0 \theta^2 \|\mathbf{B}_n\|_F^2\Big], \quad \text{for all} \quad 0< \theta \le \frac{K_1}{\|\mathbf{B}_n\|}.$$
    Here, $K_0, K_1 > 0$ are finite constants depending only on the sub-Gaussian norm.
    As $T_k$ are i.i.d. for each $k = 1, \ldots, N$, we have
    \begin{align}
        \PP\Big(\sum_{k=1}^N T_k > \delta' \Big) &\le e^{- \theta \delta'} \EE \exp \Big[\theta \sum_{k=1}^N T_k \Big] \notag \\
        &\le \exp \Big[-\theta \delta' + K_0 N\theta^2 \|\mathbf{B}_n\|_F^2 \Big], \quad \text{for} \quad 0 < \theta \le \frac{K_1}{\|\mathbf{B}_n\|}, ~ \delta' > 0. \label{eq:exponential tail bound}
    \end{align}
    Note that H\"older's inequality gives
    $$\|\bv\|^2 = \sumin v_i^2 \le \Big(\sumin v_i^r\Big)^{2/r} n^{1-2/r} = n^{1-2/r},$$
    and $|w_i|\le 1$ gives $\|\bw\|^2 = \sumin w_i^2 \le \sumin w_i^{r^\star} = 1.$\footnote{When $r = \infty$, we define $n^{1/r} = 1$. It is easy to check that the arguments in this proof does not change.}
    Combining these bounds and using $\G_n(i,j) \le 1$ give
    $$\|\mathbf{B}_n\|^2 \le \|\mathbf{B}_n\|_F^2 \le \Big(\sumin v_i^2\Big) \Big(\sumjn w_j^2\Big) \le n^{1-2/r}.$$ 
    Since we assume $N \ge n$, we can take $\theta =\frac{n^{1/r}}{\sqrt{N}} \le \frac{K_1}{\|\mathbf{B}_n\|}$, by choosing $K_1 > 1$ if necessary.
    By taking $\delta' = \Delta \sqrt{N} n^{1 - 1/r}$ for $\Delta>0$ large enough (to be chosen later), and simplifying \eqref{eq:exponential tail bound}, we get
    $$\PP\Big(\frac{1}{N}\sum_{k=1}^N T_k > \Delta \frac{n^{1-1/r}}{\sqrt{N}} \Big) \le \exp\Big[-\Delta n + K_0 n\Big].$$
\\

    \textbf{Step 3: Union bound.}
    Finally, we complete the proof by unfixing $\bv, \bw$ via a union bound:
    \begin{align*}
        \PP\Big(\|\A_n\|_{r} > \red{2} \Delta \frac{n^{1-1/r}}{\sqrt{N}}\Big) &\le \sum_{\bv \in \mathcal{N}(S_r^n, \varepsilon), ~ \bw \in \mathcal{N}(S_{r^\star}^n, \varepsilon)} \PP \left(\bv^\top \A_n \bw > \Delta\frac{n^{1-1/r}}{\sqrt{N}} \right) \\
        &\le 9^{2n}\exp(-\Delta n + K_0 n).
    \end{align*}
    Here, each inequality uses Step 1, and Step 2 combined with $\max\{|\mathcal{N}(S_r^n, \varepsilon)|, |\mathcal{N}(S_{r^\star}^n, \varepsilon)|\} \le 9^n$, respectively.
    The RHS above goes to zero when $\Delta$ is large enough, so \eqref{eq:l_n norm bound} holds.

    The claim for $r \in [1,2)$ directly follows by noting that \eqref{eq:dual formula norm} implies $\|\A_n\|_r = \|\A_n\|_{r^\star} = O_P \Big(\frac{n^{1/r}}{\sqrt{N}} \Big)$ for symmetric matrices.
\end{proof}

\begin{proof}[Proof of \cref{lem:alpha_n}]
{For notational simplicity, assume that the inverse-temperature parameter is fixed at $\theta = 1$.}
Without loss of generality, suppose $\{Z_{k,i}\}_{1\le k\le N, 1\le i\le n}$ have sub-Gaussian norms bounded by $K_0$.
Write $W_i := \frac{\sum_{k=1}^N Z_{k,i}^2}{N}$. As $\{Z_{k,i}^2\}_{1\le k\le N, 1\le i\le n}$ is a collection of sub-exponential random variables with mean 1 and sub-exponential norm $K_0^2$, Bernstein's inequality (e.g. \red{Corollary} 2.8.3 in \cite{vershynin2018high}) gives
$\PP(W_i > 2) \le e^{-K_1 N},$ for some constant $K_1 > 0$ that only depends on $K_0$. We refer the reader to Section 2 in \cite{vershynin2018high} for details regarding sub-Gaussian and sub-exponential random variables.

For a fixed $1\le i\le n$, conditioning on $(Z_{1,i},\ldots, Z_{N,i})$ (and possibly $\G_n$), we have $\A_n(i,j) = \frac{\G_n(i,j)}{N} \sum_{k=1}^N Z_{k,i} Z_{k,j}$ are independent and identically distributed for all $j \neq i$. In particular, each $\A_n(i,j) \mid (Z_{1,i},\ldots, Z_{N,i})$ has a sub-Gaussian distribution with mean $0$ and sub-Gaussian norm $\G_n(i,j) K_0 \sqrt{\frac{W_i}{N}}$. %
Hence, for any $\Theta > 0$, we have
\begin{align*}
    \PP\Big(\sqrt{n} \sum_{j \neq i} \A_n(i,j)^2 > \delta \mid &Z_{1,i},\ldots, Z_{N,i}\Big) \le e^{-\Theta \delta/\sqrt{n}} \prod_{j \neq i} \Big(\EE \Big[e^{\Theta \A_n(i,j)^2} \mid Z_{1,i},\ldots, Z_{N,i} \Big]\Big) \\
    &\le \exp \Big[ - \frac{\Theta \delta}{\sqrt{n}} + \frac{K_2^2 \Theta W_i}{N} \sum_{j \neq i} \G_n(i,j) \Big], \quad \forall 0 \le \Theta \le \frac{1}{K_2^2}\frac{N}{W_i},
\end{align*}
for some constant $K_2$ depending on $K_0$.
Here, the first inequality follows from Markov's inequality, and the second uses the sub-Gaussianity (see \cite[Proposition 2.5.2(iii)]{vershynin2018high}).
In particular, on the event $W_i \le 2$, we can take $\Theta = \frac{N}{2 K_2^2}$ and bound $\sum_{j \neq i} \G_n(i,j) \le n$ to get
$$\PP\Big(\sqrt{n} \sum_{j \neq i} \A_n(i,j)^2 > \delta  \mid Z_{1,i},\ldots, Z_{N,i}\Big) \le \exp \Big[ - \frac{N \delta}{2K_2^2\sqrt{n}} + n  \Big].$$
Finally, by a union bound and using the above upper bounds, we get
\begin{align}\label{eq:alpha_n bound i.i.d}
    \PP \Big(\maxin \sqrt{n} &\sum_{j \neq i} \A_n(i,j)^2 > \delta \Big) \le \sumin \Big[\PP \Big(\sqrt{n} \sum_{j \neq i} \A_n(i,j)^2 > \delta , ~W_i \le 2 \Big) + \PP(W_i > 2) \Big] \notag \\
    &\le \sumin \EE\Bigg[ \PP \Big(\sqrt{n} \sum_{j \neq i} \A_n(i,j)^2 > \delta \mid Z_{1,i} ,\ldots, Z_{N,i} \Big) I(W_i \le 2)\Bigg] + n e^{-K_1 N} \notag \\
    &\le n \exp \Big[ - \frac{N\delta}{2K_2^2\sqrt{n}} +n  \Big] + n e^{-K_1 N}.
\end{align}
For any constant $\delta>0$, the RHS of \eqref{eq:alpha_n bound i.i.d} is $o(1)$ as long as $n \ll N^{2/3}$. This completes the proof.
\end{proof}

\begin{proof}[Proof of Lemma \ref{lem:matrix norm}]
	
        We first show the existence of $(\I_n -\mathbf{C}_n)^{-1}$. Suppose not, i.e. there exists a vector $\bv \neq \mathbf{0}$, $\lVert \bv\rVert = 1$ such that $(\I_n -\mathbf{C}_n) \bv = \red{\mathbf{0}}$. Then, $\mathbf{C}_n \bv = \bv$, and we have a contradiction since $1 \le \lVert \mathbf{C}_n \rVert_2 \le \lVert \mathbf{C}_n \rVert_q < 1$. The second inequality follows from \red{\cref{lem:normorder}} and that $q \ge 2$.

        Now, we prove that $\lVert (\I_n -\mathbf{C}_n)^{-1} \rVert_q \le (1-{\theta})^{-1}$. For any vector $\by$ with $\lVert \by \rVert_q=1$, let $\bv := (\I_n -\mathbf{C}_n)^{-1} \by$. Then, 
		$$\lVert \by \rVert_q = \lVert (\I_n -\mathbf{C}_n) \bv \rVert_q \ge \lVert \bv \rVert_q - \lVert \mathbf{C}_n \bv \rVert_q \ge (1 - \theta) \lVert \bv \rVert_q,$$
		and we have $\lVert \bv \rVert_q \le \frac{\lVert \by \rVert_q}{1-\theta}$. Hence, $\lVert (\I_n - \mathbf{C}_n)^{-1} \rVert_q = \sup_{\lVert \by \rVert_q =1} \lVert(\I_n -\mathbf{C}_n)^{-1} \by \rVert_q \le (1-\theta)^{-1}$.
\end{proof}

\begin{acks}[Acknowledgments]
The authors would like to thank Matthias Löwe, Debdeep Pati, Roman Vershynin, and Kaizheng Wang for helpful discussions. 
\end{acks}

\bibliographystyle{imsart-number.bst}
\bibliography{reference}

\end{document}